\documentclass[11pt, one side, article]{memoir}

\settrims{0pt}{0pt} 
\settypeblocksize{*}{34pc}{*} 
\setlrmargins{*}{*}{1} 
\setulmarginsandblock{.98in}{.98in}{*} 
\setheadfoot{\onelineskip}{2\onelineskip} 
\setheaderspaces{*}{1.5\onelineskip}{*} 
\checkandfixthelayout

\usepackage{amsthm}
\usepackage{mathtools}

\usepackage[inline]{enumitem}
\usepackage{ifthen}
\usepackage[utf8]{inputenc} 
\usepackage{xcolor}

\usepackage[backend=biber, backref=true, maxbibnames = 10, style = alphabetic]{biblatex}
\usepackage[bookmarks=true, colorlinks=true, linkcolor=blue!50!black,
citecolor=orange!50!black, urlcolor=orange!50!black, pdfencoding=unicode]{hyperref}
\usepackage[capitalize]{cleveref}

\usepackage{tikz}

\usepackage{amssymb}
\usepackage{newpxtext}
\usepackage[varg,bigdelims]{newpxmath}
\usepackage{mathrsfs}
\usepackage{dutchcal}
\usepackage{fontawesome}
\usepackage{ebproof}
\usepackage{stmaryrd}
\usepackage{float}
\usepackage{lipsum}

  \crefformat{enumi}{\card#2#1#3}
  \crefalias{chapter}{section}

  \hypersetup{final}

  \setlist{nosep}
  \setlistdepth{6}

  \usetikzlibrary{ 
  	cd,
  	math,
  	decorations.markings,
		decorations.pathreplacing,
  	positioning,
  	arrows.meta,
  	shapes,
		shadows,
		shadings,
  	calc,
  	fit,
  	quotes,
  	intersections,
    circuits,
    circuits.ee.IEC
  }
  
  \tikzset{
biml/.tip={Glyph[glyph math command=triangleleft, glyph length=.95ex]},
bimr/.tip={Glyph[glyph math command=triangleright, glyph length=.95ex]},
}

\tikzset{
	tick/.style={postaction={
  	decorate,
    decoration={markings, mark=at position 0.5 with
    	{\draw[-] (0,.4ex) -- (0,-.4ex);}}}
  }
} 
\tikzset{
	slash/.style={postaction={
  	decorate,
    decoration={markings, mark=at position 0.5 with
    	{\node[font=\footnotesize] {\rotatebox{90}{$\approx$}};}}}
  }
}

\tikzset{
	oriented WD/.style={
		every to/.style={out=0,in=180,draw},
    label/.style={
    	font=\everymath\expandafter{\the\everymath\scriptstyle},
      inner sep=0pt,
      node distance=2pt and -2pt},
    semithick,
    node distance=1 and 1,
    decoration={markings, mark=at position \stringdecpos with \stringdec},
    ar/.style={postaction={decorate}},
    execute at begin picture={\tikzset{
    	x=\bbx, y=\bby,
      every fit/.style={inner xsep=\bbx, inner ysep=\bby}}}
    },
    string decoration/.store in=\stringdec,
    string decoration={\arrow{stealth};},
    string decoration pos/.store in=\stringdecpos,
    string decoration pos=.7,
    bbx/.store in=\bbx,
    bbx = 1.5cm,
    bby/.store in=\bby,
    bby = 1.5ex,
    bb port sep/.store in=\bbportsep,
    bb port sep=1.5,
    bb port length/.store in=\bbportlen,
    bb port length=4pt,
    bb penetrate/.store in=\bbpenetrate,
    bb penetrate=0,
    bb min width/.store in=\bbminwidth,
    bb min width=1cm,
    bb rounded corners/.store in=\bbcorners,
    bb rounded corners=2pt,
    bb spider/.style={
    	bb port sep=1, bb port length=10pt, bbx=.4cm, bb min width=.4cm, bby=.8ex},
    bb small/.style={
    	bb port sep=1, bb port length=2.5pt, bbx=.4cm, bb min width=.4cm, bby=.7ex},
		bb medium/.style={
			bb port sep=1, bb port length=2.5pt, bbx=.4cm, bb min width=.4cm, bby=.9ex},
    bb/.code 2 args={
    	\pgfmathsetlengthmacro{\bbheight}{\bbportsep * (max(#1,#2)+1) * \bby}
      \pgfkeysalso{draw,minimum height=\bbheight,minimum
       width=\bbminwidth,outer sep=0pt,
         rounded corners=\bbcorners,thick,
         prefix after command={\pgfextra{\let\fixname\tikzlastnode}},
         append after command={\pgfextra{\draw
            \ifnum #1=0{} \else foreach \i in {1,...,#1} {
            	($(\fixname.north west)!{\i/(#1+1)}!(\fixname.south west)$) +(-\bbportlen,0) coordinate (\fixname_in\i) -- +(\bbpenetrate,0) coordinate (\fixname_in\i')}\fi 
            \ifnum #2=0{} \else foreach \i in {1,...,#2} {
            	($(\fixname.north east)!{\i/(#2+1)}!(\fixname.south east)$) +(-
\bbpenetrate,0) coordinate (\fixname_out\i') -- +(\bbportlen,0) coordinate (\fixname_out\i)}\fi;
           }}}
		},
			bb name/.style={
     	append after command={
				\pgfextra{\node[anchor=north] at (\fixname.north) {#1};}
			}
		}
  }

\tikzset{Yonepart/.pic={
	\node[bb={1}{2},bb name = {\tiny$X_{11}$}] (X11) {};
	\node[bb={2}{2},below right=of X11,bb name = {\tiny$X_{12}$}] (X12) {};
	\node[bb={2}{1}, above right=of X12,bb name = {\tiny$X_{13}$}] (X13) {};
	\node[bb={0}{0}, inner xsep=10pt, fit={($(X11.north west)+(.3,1.5)$) (X12)  ($(X13.east)+(-.3,0)$)},bb name = {\scriptsize $Y_1$}] (Y1) {};
	\coordinate (Y1_in1') at (X11_in1-|Y1.west);
	\coordinate (Y1_in1) at (X11_in1-|Y1.west);
	\coordinate (Y1_in2') at (X12_in2-|Y1.west);
	\coordinate (Y1_in2) at (X12_in2-|Y1.west);
	\coordinate (Y1_out1') at (X13_out1-|Y1.east);
	\coordinate (Y1_out1) at (X13_out1-|Y1.east);
	\coordinate (Y1_out2') at (X12_out2-|Y1.east);
	\coordinate (Y1_out2) at (X12_out2-|Y1.east);
	\draw (Y1_in1') to (X11_in1);	
	\draw (Y1_in2') to (X12_in2);
	\draw (X11_out1) to (X13_in1);
	\draw (X11_out2) to (X12_in1);
	\draw (X12_out1) to (X13_in2);
	\draw (X12_out2) to (Y1_out2');
	\draw (X13_out1) to (Y1_out1');
	\coordinate (bottombox) at ($(X12.south)$);
	\coordinate (rightbox) at ($(X13.east)$);
	\coordinate (Y1northwest) at ($(Y1.north west)$);
	}
}
  \tikzset{Ytwopart/.pic={
	\node[bb={2}{2}, bb name = {\tiny$X_{21}$}] (X21) {};
	\node[bb={1}{2},above right=-1 and 1 of X21,bb name = {\tiny$X_{22}$}] (X22) {};
	\node[bb={0}{0}, inner xsep=10pt, fit={($(X21.south west)+(-.25,0)$) ($(X22.north east)+(.25,3.5)$)},bb name = {\scriptsize$Y_2$}] (Y2){};
	\coordinate (Y2_in1') at (X21_in2-|Y2.west);
	\coordinate (Y2_in1) at (X21_in2-|Y2.west);
	\coordinate (Y2_out1') at (X22_out2-|Y2.east);
	\coordinate (Y2_out1) at (X22_out2-|Y2.east);
	\coordinate (Y2_out2') at (X21_out2-|Y2.east);
	\coordinate (Y2_out2) at (X21_out2-|Y2.east);	
	\draw (Y2_in1') to (X21_in2);
	\draw (X21_out1) to (X22_in1);
	\draw (X22_out2) to (Y2_out1');
	\draw let \p1=(X22.south east), \p2=($(Y2_out2)$), \n1={\y1-\bby}, \n2=\bbportlen in
	  (X21_out2) to (\x1+\n2,\n1) -- (\x1+\n2,\n1) to (Y2_out2');
	\draw let \p1=(X22.north east), \p2=(X21.north west), \n1={\y1+\bby}, \n2=\bbportlen in
          (X22_out1) to[in=0] (\x1+\n2,\n1) -- (\x2-\n2,\n1) to[out=180] (X21_in1);
          }
}

\tikzset{SmallNestingPic/.pic={
  \path (0,0) pic [purple] {Yonepart};
  \path ($(rightbox)+(4,-5)$) pic [blue!40!black] {Ytwopart};
  
  \node[bb={0}{0}, fit={($(Y1northwest)+(-.5,4)$) ($(Y2.south east)+(1,0)$)}, bb name={\small $N$}] (N) {};
  \coordinate[above=\bby] (helper) at (Y2.north west);
	\coordinate (N_in1') at (Y1_in2-|N.west);
	\coordinate (N_in1) at (Y1_in2-|N.west);
	\coordinate (N_out1') at (helper-|N.east);
	\coordinate (N_out1) at (helper-|N.east);
	\coordinate (N_out2') at (Y2_out2-|N.east);
	\coordinate (N_out2) at (Y2_out2-|N.east);	  
  \draw (N_in1') to (Y1_in2);
  \draw let \p1=(Y2.north west),\p2=(Y2.north east),\n1={\y2+\bby},\n2=\bbportlen in
  (Y1_out1) to (\x1+\n2,\n1)--(\x2+\n2,\n1) to (N_out1');
  \draw (Y1_out2) to (Y2_in1);
  \draw (Y2_out2) to (N_out2');
  \draw let \p1=(Y2.north east), \p2=(Y1.north west), \n1={\y2+\bby}, \n2=\bbportlen in
          (Y2_out1) to[in=0] (\x1+\n2,\n1) -- (\x2-\n2,\n1) to[out=180] (Y1_in1);
          }
}

\theoremstyle{definition}
\newtheorem{definitionx}{Definition}[chapter]
\newtheorem{examplex}[definitionx]{Example}
\newtheorem{remarkx}[definitionx]{Remark}

\theoremstyle{plain}

\newtheorem{theorem}[definitionx]{Theorem}
\newtheorem{proposition}[definitionx]{Proposition}

\newtheorem{lemma}[definitionx]{Lemma}

\newtheorem*{theorem*}{Theorem}
\newtheorem*{proposition*}{Proposition}
\newtheorem*{corollary*}{Corollary}
\newtheorem*{lemma*}{Lemma}
\newtheorem*{warning*}{Warning}

\newenvironment{example}
  {\pushQED{\qed}\examplex}
  {\popQED\endexamplex}
  
 \newenvironment{remark}
  {\pushQED{\qed}\remarkx}
  {\popQED\endremarkx}
  
  \newenvironment{definition}
  {\pushQED{\qed}\definitionx}
  {\popQED\enddefinitionx}

\DeclareSymbolFont{stmry}{U}{stmry}{m}{n}
\DeclareMathSymbol\fatsemi\mathop{stmry}{"23}

\DeclareFontFamily{U}{mathx}{\hyphenchar\font45}
\DeclareFontShape{U}{mathx}{m}{n}{
      <5> <6> <7> <8> <9> <10>
      <10.95> <12> <14.4> <17.28> <20.74> <24.88>
      mathx10
      }{}
\DeclareSymbolFont{mathx}{U}{mathx}{m}{n}
\DeclareFontSubstitution{U}{mathx}{m}{n}
\DeclareMathAccent{\widecheck}{0}{mathx}{"71}

\renewcommand{\ss}{\subseteq}

\DeclareMathOperator{\ob}{Ob}

\DeclareMathOperator{\Funn}{Fun}
\DeclareMathOperator{\constt}{const}

\newcommand{\cat}[1]{\mathcal{#1}}
\newcommand{\Cat}[1]{\mathbf{#1}}
\newcommand{\Fun}[1]{\mathit{#1}}

\newcommand{\id}{\mathrm{id}}

\newcommand{\To}[2][]{\xrightarrow[#1]{#2}}

\newcommand{\card}{\,^{\#}}

\newcommand{\op}{^\tn{op}}

\newcommand{\tn}[1]{\textnormal{#1}}

\newcommand{\ul}[1]{\underline{#1}}

\newcommand{\nn}{\mathbb{N}}

\newcommand{\rr}{\mathbb{R}}

\newcommand{\smset}{\Cat{Set}}
\newcommand{\smcat}{\Cat{Cat}}
\newcommand{\prof}{\Cat{Prof}}

\newcommand{\yon}{\mathcal{y}}
\newcommand{\poly}{\Cat{Poly}}

\newcommand{\tri}{{\mathbin{\triangleleft}}}

\newcommand{\biglens}[2]{
     \begin{bmatrix}{\vphantom{f_f^f}#2} \\ {\vphantom{f_f^f}#1} \end{bmatrix}
}
\newcommand{\littlelens}[2]{
     \begin{bsmallmatrix}{\vphantom{f}#2} \\ {\vphantom{f}#1} \end{bsmallmatrix}
}
\newcommand{\lens}[2]{
  \relax\if@display
     \biglens{#1}{#2}
  \else
     \littlelens{#1}{#2}
  \fi
}

\newcommand{\rrnn}{\rr_{\ge 0}}

\newcommand{\poset}{\Cat{Poset}}
\newcommand{\FP}{{\Cat{Pos}}}
\newcommand{\FC}{{\Cat{Cat}}}
\newcommand{\FS}{{\Cat{sSet}}}
\newcommand{\finpos}{{\Cat{FinPos}}}
\newcommand{\emb}{\hookrightarrow}
\newcommand{\Expr}{{\Cat{Expr}}}
\newcommand{\End}{{\Cat{End}}}
\newcommand{\expr}{{\Cat{Expr}_n}}
\newcommand{\A}{{\cat{A}}}
\newcommand{\C}{{\cat{C}}}
\newcommand{\D}{{\cat{D}}}

\newcommand{\mink}{\Cat{Mink}}
\newcommand{\omink}{\overline{\mink}}
\newcommand{\minko}{\mink_1}
\newcommand{\ominko}{\omink_1}
\newcommand{\minkt}{\mink_2}
\newcommand{\ominkt}{\omink_2}

\newcommand{\dep}{\Fun{dep}}
\renewcommand{\O}{{\cat{O}}}
\renewcommand{\P}{{\cat{P}}}

\newcommand{\BV}{\mathrm{BV}}
\newcommand{\Parr}{\bindnasrepma}
\newcommand{\Par}{\;\bindnasrepma\;}

\allowdisplaybreaks
\setsecnumdepth{section}
\settocdepth{section}
\begin{document}

\title{Duoidal Structures for Compositional Dependence}

\author{Brandon T. Shapiro\thanks{Corresponding author: shapiro@topos.institute} \and David I. Spivak}

\date{}

\maketitle
\vspace{-.5cm}

\begin{abstract}
We provide a categorical framework for mathematical objects for which there is both a sort of ``independent'' and ``dependent'' composition. Namely we model them as duoidal categories in which both monoidal structures share a unit and the first is symmetric. We construct the free such category and observe that it is a full subcategory of the category of finite posets. Indeed each algebraic expression in the two monoidal operators corresponds to the poset built by taking disjoint unions and joins of the singleton poset. We characterize these ``sum-join expressible'' posets as precisely those which contain no ``zig-zags.'' We then move on to describe categories equipped with $n$-ary operations for each $n$-element finite poset; we refer to them as ``dependence categories'' since they allow for combinations of objects based on any network of dependencies between them. 

These structures model various sorts of dependence including the space-like and time-like juxtaposition of weighted probability distributions in relativistic spacetime, which we model using polynomial endofunctors on the category of sets, as well as the runtimes for multiple computer programs run in parallel and series, which we model using the tropical semiring structure on nonnegative real numbers. With these examples in mind, we conclude by describing ways in which morphisms in a partial monoidal category can be ``decorated'' in a coherent manner by objects in a dependence category, such as labeling a network of parallel programs with their runtimes.
\end{abstract}

\tableofcontents*

\chapter{Introduction}

When modeling a collection of interacting systems, one assumes---either implicitly or explicitly---a model of time. The actions of system A in the present moment affect the very \emph{possibilities} of system B in future moments. The philosophy of time is generally considered to be somewhat mysterious, whereas \emph{dependence} is a much simpler concept.

Consider the following two pictures of a $(4\times 3)$-grid, where we interpret each point as an event and we impose the \emph{dependency condition} that an event can only occur once the event directly below it has occurred:
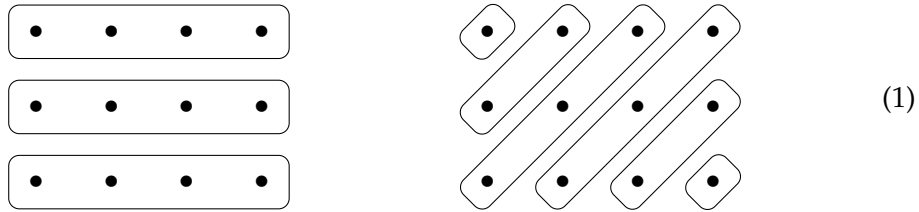
\begin{equation}\label{eqn.wentworth}
\begin{tikzpicture}[baseline=(X22)]
	\foreach \j in {1,...,3}{
		\foreach \i in {1,...,4}{
			\node at (\i,\j) (X\i\j) {$\bullet$};
		}
		\node[rounded corners, draw, fit=(X1\j) (X4\j)] (X\j) {};
	}
%
	\foreach \j in {1,...,3}{
		\foreach \i in {1,...,4}{
			\node at (6+\i,\j) (Y\i\j) {$\bullet$};
		}
	}
	\begin{scope}[every node/.style={rounded corners, draw, rotate=45}]
		\node[inner xsep=10pt, inner ysep=7pt] at ($(Y41)!.5!(Y41)$) (Y1) {};
		\node[inner xsep=30pt, inner ysep=7pt] at ($(Y31)!.5!(Y42)$) (Y2) {};
		\node[inner xsep=50pt, inner ysep=7pt] at ($(Y21)!.5!(Y43)$) (Y3) {};
		\node[inner xsep=50pt, inner ysep=7pt] at ($(Y11)!.5!(Y33)$) (Y4) {};
		\node[inner xsep=30pt, inner ysep=7pt] at ($(Y12)!.5!(Y23)$) (Y5) {};
		\node[inner xsep=10pt, inner ysep=7pt] at ($(Y13)!.5!(Y13)$) (Y6) {};
	\end{scope}
	\end{tikzpicture}
\end{equation}
In the left-hand picture, we imagine that time proceeds such that the whole bottom layer occurs simultaneously, then the middle layer, then the top layer. In the right-hand picture, we imagine that time proceeds such that the $(4,1)$-event occurs first, followed by the $(3,1)$ and $(4,2)$-events occurring simultaneously, etc. These are two incomparable dependency structures on the same set, both of which respect the imposed dependency condition.

Classically, both in physics and its various category-theoretic models (\cite{caucats}, \cite{causalstructure}, \cite{promonoidal}), the notion of dependence or independence is treated as a property of processes involving various events. For example, the theory of relativity codifies in physics the empirical logic that two objects sufficiently separated in space cannot interact without a suitable amount of time to reach each other. For example, the processes could be programs modeled as morphisms between their input and output data, where composition and the tensor product respectively represent running two programs sequentially and in parallel.

%


When two processes are temporally composed or spatially juxtaposed, this additional information often behaves in an algebraic manner. For instance, the sequential composition of two programs has a runtime given by addition, while the runtime of two programs juxtaposed in parallel is given by a maximum.

We provide here a category-theoretic account of this algebraic structure. Here, objects can be combined in two different ways---the first denoted $\otimes$ and corresponding to independence, and the second denoted $\tri$ and corresponding to dependence---and morphisms flow in the direction of increased dependence. We get operations of higher arity for each abstract arrangement of points and dependencies (arrows) between them as shown in \eqref{eqn.myposet}. The notion of dependence should not be circular, i.e.\ if $a$ depends on $b$ and $b$ depends on $a$, then $a=b$; thus we assume that each such arrangement forms a partially ordered set (poset).
\begin{figure}[h]
\begin{equation}\label{eqn.myposet}
\begin{tikzcd}
& d & \\
b \ar{ur} & & c \ar{ul} \\
& a \ar{ul} \ar{ur} \ar{uu}
\end{tikzcd}
\end{equation}
\caption{The poset corresponding to the dependence structure $a \;\tri\; (b \otimes c) \;\tri\; d$. That is $d$ depends on $b$ and $c$, which are independent of one another, and they in turn depend on $a$.}\label{fig.dependencies}
\end{figure}
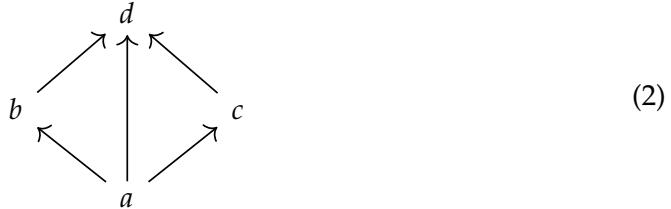


It turns out that all this is well-modeled by a refinement of \emph{duoidal categories} \cite{2monoidal,firstduoidal}, i.e.\ categories equipped with two interacting monoidal structures $\otimes,\tri$. In particular, we define what we call \emph{physical duoidal categories}, modeled after physics in the sense of 3-dimensional space and 1-dimensional time. These are simply duoidal categories for which the two units agree and for which $\otimes$ is symmetric. \cref{sec.minkowski} is devoted to showing that Minkowski spacetime is an example of a physical duoidal category, justifying the name.

The free physical duoidal category on a single generator turns out to be a full subcategory of $\poset$ spanned by what we call the \emph{expressible} posets. These are the posets that arise from algebraic expressions in $\otimes$ and $\tri$, as in \cref{eqn.myposet}, by interpreting each variable to be a singleton poset, $\otimes$ to be the sum (disjoint union) of posets, and $\tri$ to be the \emph{join} of posets (see \cref{posetduoidal}). 

Many posets, however, do not arise from these constructions, such as the poset $N$ depicted in \eqref{eqn.zigzagintro}.
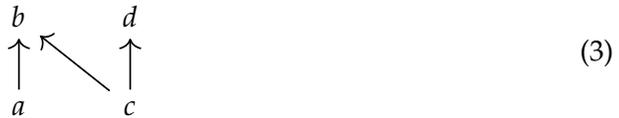
\begin{equation}\label{eqn.zigzagintro}
\begin{tikzcd}
b & d \\
a \uar & c \uar \ar{ul} 
\end{tikzcd}
\end{equation}
In fact, our first main result is that having a full copy of $N$ is the only obstruction to being expressible.

\begin{theorem*}[{\cref{expressibility}}]
A finite poset is expressible if and only if it does not contain any fully embedded copy of the poset $N$.
\end{theorem*}

Our next main result shows that a category $\cat{C}$ is compatibly equipped with an $n$-ary operation for each expressible $n$-element poset precisely when $\cat{C}$ carries the structure of a physical duoidal category.

\begin{theorem*}[{\cref{expr_operad}}]
For a category $\C$, having suitably compatible operations of the form $\C^n \to \C$ for each expressible poset on $n$ elements is equivalent to carrying the structure of a physical duoidal category, i.e.\ one in which $\otimes$ is symmetric and $\otimes,\tri$ share a unit.
\end{theorem*}

In fact, physical duoidal categories often carry additional structure of a sort that category theorists appear not to have considered. Namely, one can consider categories that can interpret \emph{every poset} on $n$ elements---not just the expressible ones---as an $n$-ary operation on $\cat{C}$. We call these \emph{dependence categories}. For example, imagine programs $a,b,c,d$ with the dependency poset depicted in \cref{eqn.zigzagintro}, where $b$ depends on $a$ and $c$, and where $d$ depends only on $c$. Then given parallel computing resources, we can create a new program that runs $a$ and $c$ in parallel, that runs $d$ as soon as $c$ finishes, and that runs $b$ as soon as $a$ and $c$ both finish. 

In many settings, a physical duoidal category can be extended to a dependence category by using limits to derive the additional operations from $\otimes$ and $\tri$.

\begin{theorem*}[{\cref{dependenceextension}}]
If $\C$ is a physical duoidal category with finite connected limits preserved by $\otimes$ and $\tri$, then $\C$ forms a dependence category.
\end{theorem*}

We formalize dependence categories using a certain \emph{categorical operad} $\FP$ of finite posets. Just as ordinary symmetric operads consist of a set of operations in each arity (as well as unit, composites, and symmetries), symmetric categorical operads use a category of operations in each arity to encode both operations and potentially non-invertible coherence morphisms between them. 

In this formalism, expressible posets form a categorical sub-operad $\expr\ss\FP$ of finite posets, which shows that any dependence category restricts to a physical duoidal category. 

We emphasize two main examples throughout the paper, motivated by modeling parallel computing and abstract events in spacetime. The first is the ``tropical'' real numbers, namely the poset of non-negative real numbers with $\otimes$ and $\tri$ given by $\max$ and $+$, modeling runtimes in parallelizable programs. The second is polynomial endofunctors on the category of sets (see \cite[Section 2.1]{aggregation} and \cref{poly_duoidal}), i.e.\ functors $\smset\to\smset$ of the form
\[ \sum_{I \in p(1)} \yon^{p[I]} \]
where $p(1)$ and each $p[I]$ are sets. Elements $I \in p(1)$ can be interpreted as possible ``outcomes'' of an event in spacetime, while the elements $i \in p[I]$ are the potential ``stimuli'' produced by the outcome $I$. An independent juxtaposition $p \otimes q$ of two polynomials has both outcomes and stimuli given by pairs of those in $p,q$, while the outcomes of a dependent composition $p \;\tri\; q$ include an outcome $I$ in $p$ and a choice of outcome in $q$ for each stimulus in $p[I]$. The stimuli can also be regarded as an un-normalized probability distribution on the set of outcomes.

We conclude by returning to the idea of processes carrying additional information. We model processes as morphisms in a partial monoidal category, and the additional information as a \emph{decoration}, assigning to each process an object in a physical duoidal category $\C$ where composition and tensor products of processes relate to $\otimes$ and $\tri$ in $\C$. In special cases of partial monoidal categories, a dependence structure on $\C$ serves as an algorithm for defining this decoration in a way that respects complex networks of processes (\cref{dependencedecoration}). In particular, we use this formalism to describe how the dependence category structure on non-negative real numbers encodes an efficient protocol for running networks of parallelizable programs (\cref{efficientparallel}).

\section*{Plan of the paper}

We begin in \cref{chap.duoidal} by introducing duoidal categories, physical duoidal categories, and our main examples of tropical reals and polynomial functors. We also begin to discuss finite posets and their relationship with free expressions in a normal duoidal category, along with a physics perspective on a subcategory of finite posets that can be realized as arrangements of points in relativistic spacetime.

In \cref{chap.expressible} we introduce expressible posets, characterize them using the obstruction pattern in \eqref{eqn.zigzagintro}, and show that free expressions in a physical duoidal category and the structure maps between them are equivalent to expressible posets. We then introduce the categorical operads of finite posets and expressible posets and show that the latter's pseudoalgebras are physical duoidal categories.

In \cref{chap.dependence} we discuss examples of dependence categories, defined as pseudoalgebras for the categorical operad of finite posets, and give conditions for extending a physical duoidal structure on a category to a dependence structure.

Finally in \cref{chap.decoration} we give category theoretic descriptions of how processes, regarded as morphisms in a partial monoidal category, can be decorated by objects in a physical duoidal category.

\section*{Acknowledgements}

Many of the ideas in this paper were inspired by conversations with John Wentworth, who first drew for us the picture in \cref{eqn.wentworth}. We would also like to thank Andre Kornell for suggesting several helpful references, Harrison Grodin for conversations which led to the inclusion of several additional examples, and Martti Karvonen for pointing out an important correction. Thanks also to Ilia Nekrasov for pointing us to a reference for ``$N$-free posets''.

This material is based upon work supported by the Air Force Office of Scientific Research under award number FA9550-20-1-0348.

\chapter{Duoidal Structures}\label{chap.duoidal}

Duoidal categories are categories with two different monoidal structures that are compatible with one another only in a certain lax manner. They arise in many different neighborhoods of mathematics, and here we show that when the monoidal units agree these categories echo principles from relativistic physics by encoding how events can be juxtaposed with different choices of causal dependencies. These dependencies are formalized as finite posets, whose combinatorics encode an elegant description of the structure of such a duoidal category.

\section{Duoidal Categories}

Duoidal categories were first defined in \cite[Definition 6.1]{2monoidal} and have at times been referred to as ``2-monoidal'' categories.

\begin{definition}[Duoidal category]
A \emph{duoidal category} consists of a category $\C$ equipped with two monoidal structures $(\yon_\otimes,\otimes)$ and $(\yon_\tri,\tri)$ such that the functors $\tri\colon \C \times \C \to \C$ and $\yon_\tri\colon 1 \to \C$ 
are lax monoidal with respect to $(\yon_\otimes,\otimes)$ on $\C$ and $(\yon_\otimes \times \yon_\otimes,\otimes \times \otimes)$ on $\C \times \C$, compatibly with the coherence isomorphisms for $(\yon_\tri,\tri)$.%
\footnote{We use the notation $\otimes,\tri,\yon$ from our main example of polynomial functors on $\smset$, as discussed in \cref{poly_duoidal}.}
\end{definition}

Alternatively, duoidal categories can be defined by the two monoidal structures along with the generating structure maps
\begin{equation}\label{eqn.duoidal}
(a \;\tri\; b) \otimes (c \;\tri\; d) \to (a \otimes c) \;\tri\; (b \otimes d) \qquad\quad \yon_\tri \otimes \yon_\tri \to \yon_\tri \qquad\quad \yon_\otimes \to \yon_\otimes \;\tri\; \yon_\otimes \qquad\quad \yon_\otimes \to \yon_\tri
\end{equation}
natural in $a,b,c,d$ which satisfy equations ensuring that they commute in a suitable sense with the associators and unitors of the monoidal structures. 
The morphism on the left in \eqref{eqn.duoidal} is called the \emph{lax interchanger}.

\begin{example}[Coproducts, products]\label{cocartmonoidal}
Let $(\C,\yon_\tri, \tri)$ be any monoidal category with finite coproducts (not necessarily respected by $\tri$ in any way). Then $(\varnothing, \sqcup)$ and $(\yon_\tri, \tri)$ form a duoidal structure on $\C$, with the maps
\[
(a \;\tri\; b) \sqcup (c \;\tri\; d) \to (a \sqcup c) \;\tri\; (b \sqcup d) \qquad\quad \yon_\tri \sqcup \yon_\tri \to \yon_\tri \qquad\quad \varnothing \to \varnothing \;\tri\; \varnothing \qquad\quad \varnothing \to \yon_\tri
\] 
induced by the universal property of coproducts and initial objects.

Dually, in any monoidal category $(\C,\yon_\otimes, \otimes)$ with finite products, the structures $(\yon_\otimes, \otimes)$ and $(1, \times)$ form a duoidal structure on $\C$.
\end{example}

\begin{example}[Braided monoidal categories]
If $\C$ is a braided monoidal category, then two copies of its monoidal structure form a duoidal structure, where the three rightmost maps in \eqref{eqn.duoidal} are identities and the lax interchange map on the left is given by applying the braiding to $b$ and $c$.
\end{example}

\begin{example}[Tropical reals]\label{tropical}
The poset of non-negative real numbers $\rrnn$ with the usual order can be regarded as a category, with two monoidal structures given by $(0,\max)$ and $(0,+)$. These two operations make $\rrnn$ a duoidal category, since the units agree and 
\[ (a + b) \max {(c + d)} \le (a \max c) + (b \max d) \]
for all $a,b,c,d \in \rrnn$. 

This example, inspired by conversations with Harrison Grodin in relation to \cite{costaware}, is motivated by the analysis of runtime in parallel programming: the runtime of two programs run in parallel is the maximum of the two runtimes, while the runtime of two programs run in series is the sum of their runtimes. The lax interchanger corresponds to the observation that given four programs, running two sequential pairs in parallel is generally faster than waiting for both of the first two parallel programs to finish before starting either of the second two.
\end{example}

\begin{example}[Polynomial functors]\label{poly_duoidal}
Let $\poly$ denote the category of polynomial endofunctors on $\smset$ and natural transformations between them. A polynomial functor has the form
\[ p = \sum_{I \in p(1)} \; \prod_{i \in p[I]} \; \yon \; =: \sum_{I \in p(1)} \; \yon^{p[I]}, \]
where $p(1)$ and each $p[I]$ are sets, $\sigma$ denotes sum (disjoint union), $\prod$ denotes cartesian product, $\yon$ is the identity functor, and $\yon^{p[I]}$ is the functor $\smset(p[I],-)\colon\smset \to \smset$ represented by the set $p[I]$. 

There are many duoidal structures on $\poly$ by \cref{cocartmonoidal} as $\poly$ has (co)products inherited from the functor category $\Funn(\smset,\smset)$, but we will focus on a duoidal structure involving neither the product nor coproduct. The \emph{Dirichlet tensor product} (see \cite[Proposition 2.1.11]{aggregation}) sends a pair of polynomials $p$ and $q$ to
\[ p \otimes q = \sum_{\substack{I \in p(1) \\ J \in q(1)}} \; \prod_{\substack{i \in p[I] \\ j \in q[J]}} \; \yon \; = \sum_{\substack{I \in p(1) \\ J \in q(1)}} \; \yon^{p[I] \times q[J]} \]
while the \emph{composition product} sends $p$ and $q$ to their composite as endofunctors on $\smset$, resulting in the polynomial
\[ p \;\tri\; q = \sum_{I \in p(1)} \; \prod_{i \in p[I]} \; \sum_{J \in q(1)} \; \prod_{j \in q[J]} \; \yon \; = \sum_{\substack{I \in p(1) \\ f \colon p[I] \to q(1)}} \; \prod_{\substack{i \in p[I] \\ j \in q[fi]}} \; \yon \; = \sum_{\substack{I \in p(1) \\ f \colon p[I] \to q(1)}} \; \yon^{\sum\limits_{i \in p[I]} q[fi]} \]
which agrees with the classical composition of polynomials.

Both products $\otimes$ and $\tri$ form monoidal structures, the former symmetric, with the identity polynomial $\yon$ as the unit. As shown in \cite[Proposition 2.1.14]{aggregation}, these two monoidal structures make $\poly$ a duoidal category, with the lax interchanger given by the natural transformation
\[
\begin{array}{ccccccccc}
(p \;\tri\; q) \otimes (r \;\tri\; s) & = &
\! \! \displaystyle\sum_{\substack{I \in p(1) \\ J \colon p[I] \to q(1)}} \! \! &
\! \! \displaystyle\sum_{\substack{K \in r(1) \\ L \colon r[K] \to s(1)}} \! \! &
\! \! \displaystyle\prod_{\substack{i \in p[I] \\ j \in q[Ji]}} \! &
\! \! \displaystyle\prod_{\substack{k \in r[K] \\ \ell \in s[Lk]}} \! &
\yon \\ \\
& = & 
\! \! \displaystyle\sum_{\substack{I \in p(1) \\ K \in r(1)}} \! \! &
\! \! \displaystyle\sum_{\substack{J \colon p[I] \to q(1) \\ L \colon r[K] \to s(1)}} \! \! &
\! \! \displaystyle\prod_{\substack{i \in p[I] \\ k \in r[K]}} \! &
\! \! \displaystyle\prod_{\substack{j \in q[Ji] \\ \ell \in s[Lk]}} \! &
\yon \\ \\
& \to & 
\! \! \displaystyle\sum_{\substack{I \in p(1) \\ K \in r(1)}} \! \! &
\! \! \displaystyle\sum_{\substack{R \colon p[I] \times r[K] \\ \to q(1) \times s(1)}} \! \! &
\! \! \displaystyle\prod_{\substack{i \in p[I] \\ k \in r[K]}} \! &
\! \! \displaystyle\prod_{\substack{j \in q[\pi_1R(i,k)] \\ \ell \in s[\pi_2R(i,k)]}} \! &
\yon & 
= & (p \otimes r) \;\tri\; (q \otimes s),
\end{array}
\]
where the last map takes $(J \colon p[I] \to q(1), L \colon r[K] \to s(1))$ to the function 
\[ J \times L \colon p[I] \times r[K] \\ \to q(1) \times s(1). \]
\end{example}

\section{Physical duoidal categories}

The duoidal categories we are primarily interested in are those for which the units of $\otimes$ and $\tri$ agree, where we will denote both by $\yon$. In this case all but the interchanger in \eqref{eqn.duoidal} are isomorphisms (though we will sometimes refer to them as if they were identities), and we obtain many new structure maps that are not present in general duoidal categories. The most important is what we call the \emph{comparitor}:
\begin{equation}\label{eqn.comparitor}
a \otimes b = (a \;\tri\; \yon) \otimes (\yon \;\tri\; b) \to (a \otimes \yon) \;\tri\; (\yon \otimes b) = a \;\tri\; b
\end{equation}
but there are many other new structure maps as well, e.g.\
\begin{equation}\label{eqn.interleave}
(a \;\tri\; b) \otimes (c \;\tri\; d) = (a \;\tri\; b \;\tri\; \yon) \otimes (\yon \;\tri\; c \;\tri\; d) \to (a \otimes \yon) \;\tri\; (b \otimes c) \;\tri\; (\yon \otimes d) = a \;\tri\; (b \otimes c) \;\tri\; d
\end{equation}
\begin{align*}\label{eqn.unwind}
(a \;\tri\; b) \otimes (c \;\tri\; d) &= 
(\yon \;\tri\; \yon \;\tri\; a \;\tri\; b) \otimes (c \;\tri\; d \;\tri\; \yon \;\tri\; \yon)\\&\to 
(\yon \otimes c) \;\tri\; (\yon \otimes d) \;\tri\; (a \otimes \yon) \;\tri\; (b \otimes \yon) \to c \;\tri\; d \;\tri\; a \;\tri\; b
\end{align*}

\begin{definition}[Normal duoidal category]
A duoidal category is called \emph{normal} when the lax monoidal functors $\tri \colon \C \times \C \to \C$ and $\yon_\tri \colon 1 \to \C$ are \emph{normalized}, which is to say, preserve units up to coherent isomorphism with respect to $(\yon_\otimes,\otimes)$.
\end{definition}

This definition (see also \cite[Definition 4]{commutativity}) corresponds precisely to the condition that the right three morphisms in \eqref{eqn.duoidal} are isomorphisms, though it is in fact sufficient to impose only that $\yon_\otimes \cong \yon_\tri$.

We will further restrict to the case in which $\otimes$ is symmetric, which in our physical interpretation (see \cref{sec.minkowski}) corresponds to the dimensionality of space being at least 3. As time will be treated as 1-dimensional, $\tri$ will not be assumed to be symmetric. 

\begin{definition}[Physical duoidal category]\label{def.physical}
A \emph{physical} duoidal category is a normal duoidal category $\C$ in which the monoidal structure $(\yon,\otimes)$ is symmetric and the symmetry isomorphisms commute with the interchangers in the sense that the diagram in \eqref{eqn.interchangesymmetry} commutes for all $a,b,c,d$ in $\C$.
\end{definition}

\begin{equation}\label{eqn.interchangesymmetry}
\begin{tikzcd}
(a \;\tri\; b) \otimes (c \;\tri\; d) \rar \dar[swap]{\cong} & (a \otimes c) \;\tri\; (b \otimes d) \dar{\cong} \\
(c \;\tri\; d) \otimes (a \;\tri\; b) \rar & (c \otimes a) \;\tri\; (d \otimes b) 
\end{tikzcd}
\end{equation}

\begin{example}[Tropical reals]
As 0 is the unit for both $\max$ and $+$, $\rrnn$ is physical duoidal. The \emph{comparitor} morphism from \eqref{eqn.comparitor} is the standard inequality $a \max {b} \le a + b$.
\end{example}

\begin{example}[Polynomial functors]
The duoidal category of polynomial functors is physical as $\otimes$ and $\tri$ share the same unit, $\yon$, and $\otimes$ is symmetric. For polynomials $p,q$, the natural transformation 
\[ p \otimes q = \sum_{\substack{I \in p(1) \\ J \in q(1)}} \; \prod_{\substack{i \in p[I] \\ j \in q[J]}} \; \yon \; \to \sum_{\substack{I \in p(1) \\ f \colon p[I] \to q(1)}} \; \prod_{\substack{i \in p[I] \\ j \in q[fi]}} \; \yon \; = p \;\tri\; q \]
from \eqref{eqn.comparitor} sends $(I,J)$ to $(I,\constt_J \colon p[I] \to q(1))$.
\end{example}

\begin{example}[BV Categories]\label{bvcats}
$\BV$-categories were introduced in \cite{bvprob} to model an extension of multiplicative linear logic called $\BV$, which introduces a noncommuting connective corresponding to sequential combination. Multiplicative linear logic is modeled by symmetric linear distributive categories, namely categories with two symmetric monoidal structures $(\top,\otimes)$ and $(\bot,\Parr)$ along with natural maps
\begin{equation}\label{eqn.distributivity}
a \otimes (b \Par c) \to (a \otimes b) \Par c \quad \textrm{and} \quad (a \Par b) \otimes c \to a \Par (b \otimes c) 
\end{equation}
satisfying various equations. A $\BV$-category is a symmetric linear distributive category equipped with an additional monoidal structure $(I,\varoslash)$ such that $(\top,\otimes)$ and $(I,\varoslash)$ form a normal (and hence physical) duoidal structure while $(I,\varoslash)$ and $(\bot,\Parr)$ are related by a lax interchanger satisfying a compatibility equation (see \cite[Theorem 3.3]{bvprob}). 

While we will not discuss them further, $\BV$-categories arise in multiple physically-motivated settings. In \cite[Section 5]{bvprob}, a $\BV$ structure is described for Girard's probabilistic coherence spaces, and \cite{causalbv} shows that a higher order causal theory (as first defined in \cite{causalstructure}) forms a $\BV$-category.
\end{example}

\begin{example}[Posets]\label{posetduoidal}
The category $\poset$ of posets forms a physical duoidal category, with $\yon\coloneqq\varnothing$ the empty poset, with $\otimes\coloneqq\sqcup$ the sum of posets, and with $\tri\coloneqq\Join$ given by the \emph{join} operation: for posets $P$ and $Q$, $P \Join Q$ is the poset with underlying set $P \sqcup Q$, where $x \le y$ when either $x \le_P y$, $x \le_Q y$, or $x \in P$ with $y \in Q$. Graphically, we can draw the relations in a poset as arrows in the corresponding category, with the posets $P \sqcup Q$ and $P \Join Q$ represented as the left and right diagrams respectively in \eqref{eqn.posetjoin}.
\begin{equation}\label{eqn.posetjoin}
\left(\begin{tikzcd}
P & Q
\end{tikzcd}\right)\qquad\qquad\left(\begin{tikzcd}
Q \\ P \uar
\end{tikzcd}\right)
\end{equation}
In the graphical representation of $P \Join Q$, the single arrow from $P$ to $Q$ denotes an arrow from each element of $P$ to each element of $Q$, as in the definition of the poset $P \Join Q$.

It is straightforward to check that $\varnothing$ is a unit for both $\sqcup$ and $\Join$, and the interchanger 
\[ (P \Join Q) \sqcup (R \Join S) \to (P \sqcup R) \Join (Q \sqcup S) \]
is an identity-on-elements inclusion, evident from the graphical representation in \eqref{eqn.posetinterchanger}.
\begin{equation}\label{eqn.posetinterchanger}
\left(\begin{tikzcd}
Q & S \\
P \uar & R \uar
\end{tikzcd}\right)\quad\emb\quad\left(\begin{tikzcd}
Q & S \\
P \uar \ar{ur} & R \uar \ar{ul}
\end{tikzcd}\right)
\end{equation}
The comparitor $P \sqcup Q \to P \Join Q$ is the identity-on-objects inclusion evident from \eqref{eqn.posetjoin}, and furthermore as the same is true for the interchanger all duoidal structure maps between posets will be identity-on-objects.
\end{example}

The diagrams in \cref{posetduoidal} provide a way to visualize all of the structure maps in a physical duoidal category as inclusions of posets.  In \eqref{eqn.posetjoin} and \eqref{eqn.posetinterchanger}, the symbols $P,Q,R,S$ can be interpreted as single elements rather than entire posets (or each as the singleton poset), so that each diagram represents a unique poset. For instance, the structure maps in \eqref{eqn.interleave} are represented by the left and right identity-on-elements poset inclusions in \eqref{eqn.interleavediagram}.
\begin{equation}\label{eqn.interleavediagram}
\left(\begin{tikzcd}
b & d \\
a \uar & c \uar
\end{tikzcd}\right)\quad\emb\quad\left(\begin{tikzcd}
b \rar & d \\
a \uar \rar & c \uar
\end{tikzcd}\right)\qquad\qquad\left(\begin{tikzcd}
b & d \\
a \uar & c \uar
\end{tikzcd}\right)\quad\emb\quad\left(\begin{tikzcd}
b & d \ar{dl} \\
a \uar & c \uar
\end{tikzcd}\right)
\end{equation}

In \cref{chap.expressible}, we characterize the posets that arise from $\sqcup$ and $\Join$ and show that the identity-on-elements morphisms between such posets correspond precisely to the structure maps in a physical duoidal category. This way structure maps can not only be modeled as poset inclusions but also recognized from them, such as the linear distributivity maps in \eqref{eqn.distributivity} (with $\tri$ in place of $\Parr$) which are represented by the poset inclusions in \eqref{eqn.distinclusions}.
\begin{equation}\label{eqn.distinclusions}
\left(\begin{tikzcd}
 & c \\
a & b \uar
\end{tikzcd}\right)\quad\emb\quad\left(\begin{tikzcd}
 & c \\
a \ar{ur} & b \uar
\end{tikzcd}\right)\qquad\qquad\left(\begin{tikzcd}
b & c \\
a \uar 
\end{tikzcd}\right)\quad\emb\quad\left(\begin{tikzcd}
b & c \\
a \uar \ar{ur}
\end{tikzcd}\right)
\end{equation}

\section{Minkowski spacetime}\label{sec.minkowski}

For an example inspired by classical physics, we describe a category whose objects are arrangements of points in 4-dimensional space, which we regard as Minkowski spacetime. In Minkowski spacetime, there is a dependency relation in which a point $p$ is dependent on another point $q$ if $p$ is in the ``light cone'' of $q$, meaning it is reachable from $q$ by traveling no greater than the speed of light.

\begin{definition}[Minkowski dependence]
The \emph{Minkowski causal dependence} (or simply \emph{Minkowski dependence}) relation on $\rr^4$ has $q \le p$ if $r\coloneqq p-q$ is \emph{timelike} or \emph{lightlike},\footnote{In the theory of relativity, a vector such as $r$ is timelike if it represents a path through spacetime whose speed is below the speed of light, lightlike if the speed is equal to the speed of light, and spacelike if the speed is greater than the speed of light. Assuming one has access to light speed, the first two are the paths that are physically traversible and hence relevant when discussiong causality.} meaning that 
\[c\;\sqrt{r_1^2 + r_2^2 + r_3^2}\;\leq r_4.\] 
Here $c$ denotes the speed of light and $(r_1,r_2,r_3,r_4)$ are the coordinates of $r$. 
\end{definition}

\begin{definition}
Define the category $\mink$ as follows:
\begin{itemize}
	\item its objects are pairs $(X,p)$ where $X$ is a finite set and $p \colon X \to \rr^4$ is any injective function; 
	\item its morphisms $(X,p) \to (Y,q)$ are given by functions $f \colon X \to Y$ such that if $p(a)$ is Minkowski dependent on $p(b)$ for $a,b \in X$, then $q(f(a))$ is Minkowski dependent on $q(f(b))$;
	\item identities and composites are given by identities and composites of functions, all of which preserve the Minkowski-dependence. 
\qedhere
\end{itemize}
\end{definition}

While this definition is where the physical intuition comes from, $\mink$ is in fact equivalent to a full subcategory of $\finpos$, the category of finite posets and monotone maps, because the morphisms in $\mink$ only use the Minkowski dependence partial order on $\rr^4$, not all the data of the point-embeddings. The functor $\dep \colon \mink \to \finpos$ sends $(X,p)$ to the poset structure on $X$ given by the Minkowski-dependence relation induced by $p$, and sends morphisms $f \colon (X,p) \to (Y,q)$ to the underlying function $X \to Y$, whose monotonicity is equivalent to the Minkowski dependence-preservation condition. This functor is clearly fully faithful, but not necessarily essentially surjective as that would require that every finite poset have a monotone embedding into Minkowski space.

Embeddings of posets into Minkowski space have been studied in \cite{minkdim}, where it is observed that posets which embed into Minkowski space with 3 spacial dimensions and 1 time dimension are precisely those which correspond to the inclusion order on some arrangement of filled-in 2-spheres in 3-dimensional Euclidean space. This is because for any points $p,q$ in Minkowski space with $p \le q$, the light cones of $p$ and $q$ intersect the 3-d plane at any fixed future time coordinate as a pair of filled-in spheres, the one for $q$ inside the one for $p$.

In \cite[Theorem 2.1]{sphereorders}, however, it is shown that for sufficiently large $n$, the $n \times n \times n$ 3-dimensional grid poset cannot be modeled as an inclusion order on a sphere arrangement, not just in 3-d space but in any dimension. This means that $\dep$ is not essentially surjective. But while $\mink$ does not include the entire category of finite posets, it does have enough to retain a physical duoidal structure.

\begin{proposition}\label{minkduoidal}
$\mink$ forms a physical duoidal subcategory of finite posets.
\end{proposition}

\begin{proof}
To show this, it suffices to show that the posets which embed into Minkowski space include the singleton poset (which is evidently true) and are closed under disjoint union (sum) and join. For disjoint union, observe that for any two posets embedded into Minkowski space, moving them sufficiently far apart in the spacial direction will make them disjointly separated with respect to causal dependence. From the perspective of sphere arrangements, this is even simpler: the disjoint union of sphere arrangements represents the disjoint unions of their corresponding posets. For joins, observe that if any two posets are moved sufficiently far apart in the time direction the later one will eventually be contained in the intersected light cones of the earlier one, providing the relations between the two present in their join. To see this using sphere arrangements, note that the containment order of a sphere arrangement is unaffected by uniformly enlarging all of the spheres relative to their centers. When the spheres of the earlier arrangement have been sufficiently expanded, they will have a nonempty intersection into which the later arrangement can be embedded.
\end{proof}

While the proof of \cref{minkduoidal} defines the duoidal structure on $\mink$ relative to that on posets, note that defining $\otimes$ and $\tri$ directly on $\mink$ does not require any more specificity. The operation $\otimes$ juxtaposes two arrangements of points so that they are separated spacially, while the operation $\tri$ juxtaposes them so that they are separated in time with all possible dependencies between them. Because the morphisms in $\mink$ ``see'' only the dependencies of the points in the arrangements (in the sense that any two arrangements with the same dependence structure are uniquely isomorphic), any specific formulas for how arrangements are to be juxtaposed in space or time would be uniquely isomorphic so long as they meet the stated dependence conditions, and the unitor and associator isomorphisms are uniquely determined. This means that it is truly only necessary to observe that any two arrangements \emph{can} be juxtaposed entirely spacially or temporally.

Interestingly, this is not the case for Euclidean spacetime, where a point $p$ can be dependent on $q$ so long as $p$ has a larger time coordinate than $q$. This is because for any two arrangements in Euclidean spacetime which each contain points at different time coordinates, there is no way to juxtapose them in space that preserves their own dependencies without introducing any new ones between the two. Only in relativistic spacetime, where dependence requires a path through spacetime not exceeding the speed of light, can two such arrangements be guaranteed to be causally independent.

A more detailed category modeling Minkowski space might ask for morphisms between arrangements of points to also include paths between those points. This perspective will help illustrate how the symmetry of $\otimes$, and lack of symmetry of $\tri$, in the definition of physical duoidal category is linked to the setting of 3-dimensional space and 1-dimensional time.

\begin{definition}\label{fpath}
Given a morphism $f \colon (X,p) \to (Y,q)$ in $\mink$, an \emph{$f$-path} is a continuous function $\gamma \colon X \times I \to \rr^4$, where $X$ is regarded as a discrete space and $I$ is the unit interval, such that
\begin{itemize}
	\item for each $x \in X$ the restriction $\gamma_x \colon I \xrightarrow{x \times \id} X \times I \xrightarrow{\gamma} \rr^4$ is an embedded path from $p(x)$ to $q(f(x))$; and
	\item for each $i_1 \le j_1$ and $i_2 \le j_2$ in $I$ and $x_1,x_2 \in X$, if $\gamma(x_1,i_1) \le \gamma(x_2,i_2)$ then $\gamma(x_1,j_1) \le \gamma(x_2,j_2)$, where $\leq$ between points in $\rr^4$ denotes Minkowski dependence.
\qedhere
\end{itemize}
\end{definition}

The dependence condition can be viewed as imposing that the dependencies in $(X,p)$ are preserved at every stage on the path to $(Y,q)$, even if $x_1$ and $x_2$ traverse their paths at different rates. Furthermore, while $\gamma$ need not be an embedding, by the dependence condition no two paths $\gamma_{x_1}$ and $\gamma_{x_2}$ will intersect if $f(x_1) \neq f(x_2)$. 

\begin{definition}\label{pathmink}
The category $\overline{\mink}$ has the same objects as $\mink$, but a morphism from $(X,p)$ to $(Y,q)$ consists of a dependence-preserving function $f \colon X \to Y$ as well as an isotopy class $f$-paths. The identity on $(X,p)$ is given by the identity function on $X$ and the class of the $\id$-path $X \times I \xrightarrow{\pi_1} X \xrightarrow{p} \rr^4$. For $f \colon (X,p) \to (Y,q)$ and $g \colon (Y,q) \to (N,r)$ in $\mink$, composition of an $f$-path $\gamma$ and a $g$-path $\rho$ is the $g \circ f$-path which when restricted to $x \in X$ is the concatenation of $\gamma_x$ with $\rho_{f(x)}$. This composition preserves isotopy and is unital and associative up to isotopy, endowing $\omink$ with the structure of a category.
\end{definition}

Any morphism $f$ in $\mink$ has an $f$-path, and as any two tangles of paths in 4-dimensional space are related by isotopy this $f$-path is unique, so the data of these paths does not change the same category $\mink$.

\begin{proposition}\label{3dequiv}
$\omink$ is equivalent to $\mink$.
\end{proposition}

However, \cref{3dequiv} is reliant on facts about 4-dimensional spacetime which do not hold in lower dimensions.

\begin{definition}
Let $\minko$ (resp. $\minkt$) be the analogous categories to $\mink$ in which spacetime has only 1 (resp. 2) spatial dimensions. Similarly let $\ominko$ (resp. $\ominkt$) be the analogue of $\omink$ with 1 (resp. 2) spatial dimensions.
\end{definition}

We can now describe how \cref{3dequiv} does not hold in lower dimensions, where relatedly spacial juxtaposition $\otimes$ is not symmetric in $\ominko$ or $\ominkt$.

\begin{example}\label{1obstruction}
In 1-dimensional space, consider $(X,p)$ with $X = \{x_1,x_2\}$, $p(x_1) = (1,0)$ and $p(x_2) = (-1,0)$, namely two points separated in space with no causal dependence, and $Y = \{y_1,y_2\}$ with $q(y_1) = (-1,1)$ and $q(y_2) = (1,1)$. The function $f \colon X \to Y$ sending $x_1$ to $y_1$ and $x_2$ to $y_2$ is a morphism in $\mink_1$, as there are no dependencies between the points in $X$. But any pair of disjoint paths from $p(x_1)$ to $q(y_1)$ and from $p(x_2)$ to $q(y_2)$ which do not intersect will both at some point cross the time axis in $\rr^2$. At these points in the paths, there is causal dependence in one direction or the other which is no longer present at the end of the paths, violating the dependence condition in \cref{fpath}. There is therefore no morphism in $\ominko$ with $f$ as its underlying function. 

This obstruction is closely related to the fact that two points on the interval cannot move along the interval into each other's positions without at some point intersecting one another.
\end{example}

Given an object $(X,p)$ in $\ominko$, there is a ``spatial'' linear order on the set of connected components of the associated dependence poset given by the spatial coordinate, as any connected component must be an adjacent block in the linear order on $X$ induced by the spatial coordinate. Based on the obstruction in \cref{1obstruction} to finding an $f$-path when $f$ reverses the spatial order, $\ominko$ is equivalent to the subcategory of $\mink_1$ containing only the functions $f \colon (X,p) \to (Y,q)$ which preserve not only the dependence order but also the spatial order on the dependence-connected components. This subcategory inherits the structure of a normal duoidal category, but $\otimes$ is no longer symmetric as the spacial juxtapositions of two singleton points in either order are isomorphic to the arrangements in \cref{1obstruction}. 

\begin{proposition}
$\ominko$ is a normal duoidal category in which $\otimes$ is not symmetric.
\end{proposition}

Finally, in 2-dimensional space, the obstruction to symmetry of spatial juxtaposition in \cref{1obstruction} is not present, as two points separated in space can be swapped along the sides of a circle in 2-dimensional space without introducing any additional dependence. However, these paths are no longer unique up to isotopy, as for instance a path of one point circling around another in space cannot be deformed to the constant path without passing above or below (in time) the second point. This leads to the following observation.

\begin{proposition}
$\ominkt$ is a normal duoidal category in which $\otimes$ is braided.
\end{proposition}

The dimensions of space then correspond to the levels of symmetry in monoidal categories, with 1 dimension allowing for monoidal structure, 2 dimensions allowing for braiding, and 3 dimensions admitting symmetry. If we instead modeled arrangements in Minkowski space using higher dimensional categories, this classification would continue above dimension 3, but in this framework adding any additional dimensions keeps $\otimes$ symmetric.

Thus we see that our definition (\ref{def.physical}) of physical duoidal category evokes a notion of spacetime in which time can be ordered, but where space is totally unordered in the sense that it is at least 3-dimensional: the symmetry of $\otimes$ means that things can move around each other without getting tangled up.

\chapter{Sum-join expressible posets}\label{chap.expressible}

In the diagrams \eqref{eqn.posetjoin}, \eqref{eqn.posetinterchanger}, and \eqref{eqn.interleavediagram}, an algebraic expression composed of the binary operations $\otimes$ and $\tri$ is converted into a poset by treating each variable as the singleton poset on that letter, $\otimes\coloneqq\sqcup$ as the sum (disjoint union) of posets, and $\tri\coloneqq\Join$ as the join of posets. Structure maps between two of these expressions in a physical duoidal category are then observed to correspond to an identity-on-elements inclusion of the corresponding posets. 

This correspondence provides a graphical formalism to reason about the structure of a physical duoidal category, and in this chapter we analyze the posets that arise from this correspondence and prove that the identity-on-elements inclusions between them agree precisely with the physical duoidal structure maps between the analogous algebraic expressions.

\begin{definition}[Sum-join expressible poset]\label{def.expressible}
A finite poset is \emph{sum-join expressible} if it is either empty or constructible out of singleton posets using only joins and sums (i.e.\ disjoint unions).
\end{definition}

Sum-join expressible posets have been studied under many names: \emph{$N$-free posets} \cite{cameron1987some}, \emph{reticles} \cite{schmerl1980decidability}, and \emph{TSP digraphs} \cite{lawler1978sequencing}.

The definition of sum-join expressible posets is \emph{inductive} in nature, in the sense that it is of the following form.
\begin{itemize}
	\item The empty poset $\varnothing$ is sum-join expressible;
	\item any singleton poset $\{a\}$ is sum-join expressible;
	\item if posets $P,Q$ are sum-join expressible, so is $P \sqcup Q$; and
	\item if posets $P,Q$ are sum-join expressible, so is $P \Join Q$.
\end{itemize}
Thus certain posets are realized via an expression in the language of variables (representing the singleton poset on an element with the same name as the variable), $\yon$ (representing the empty poset $\varnothing$), and the binary operations $\otimes$ and $\tri$ (representing $\sqcup$ and $\Join$ respectively). So, for instance, the poset depicted in \eqref{eqn.expression} can be expressed as $a \;\tri\; (b \otimes (c \;\tri\; d))$.
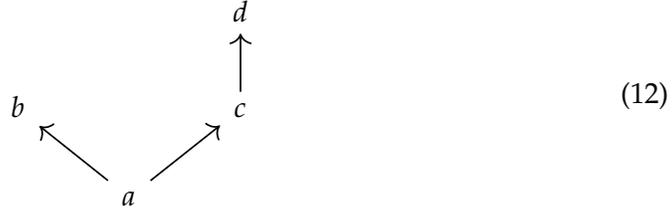
\begin{equation}\label{eqn.expression}
\begin{tikzcd}[row sep=small]
& & d \\
b & & c \uar \\
& a \ar{ul} \ar{ur}
\end{tikzcd}
\end{equation}

\section{Concrete Characterization}

Not every finite poset is sum-join expressible; for example, consider the zig-zag poset depicted in \eqref{eqn.zigzag}.
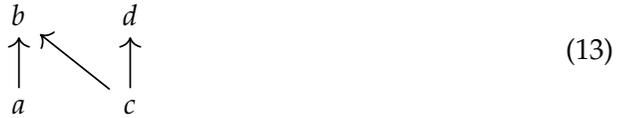
\begin{equation}\label{eqn.zigzag}
N\coloneqq\fbox{
\begin{tikzcd}[ampersand replacement=\&]
b \& d \\
a \uar \& c \uar \ar{ul} 
\end{tikzcd}
}
\end{equation}
We refer to this poset as $N$; it is nonempty, not a sum of nonempty posets as it is connected, and also not a join of nonempty posets as there is no partition of its vertices for which there is an arrow from each vertex in the first to each vertex in the second. Therefore $N$ is not sum-join expressible. It is natural to ask, then: which finite posets are sum-join expressible and which are not?

The inductive form of \cref{def.expressible} is helpful for turning sum-join expressions into posets, but does not provide much guidance for how to distinguish sum-join expressible posets from arbitrary finite posets. A more helpful description in practice provides a concrete and (ideally) efficient strategy for checking whether a given poset is sum-join expressible without any prior information about how it was built. One way to do this is to identify a complete set of patterns that prevent a poset from being sum-join expressible, so that any poset without any such obstructions is always sum-join expressible.

To this end, we now show that there is a straightforward algorithm for detecting whether a given finite poset is sum-join expressible, which is constructive in the sense of providing a construction of the corresponding sum-join expression. As it happens, the \emph{non}-sum-join-expressible pattern $N$ from \eqref{eqn.zigzag} is the only obstruction to sum-join expressibility.

\begin{theorem}\label{expressibility}
A finite poset is sum-join expressible if and only if it has no full embedding of $N$.
\end{theorem}

Here a full embedding of $N$ in a poset $P$ is a monotone map $N \to P$ which is injective on elements and reflects order in addition to preserving it. In other words, it means $P$ contains four distinct elements whose partial order inherited from $P$ is isomorphic to $N$.

\cref{expressibility} explains the name $N$-free \cite{cameron1987some} for sum-join expressible posets.

\begin{proof}
We first show by induction that a sum-join expressible poset has no full embeddings of $N$. Empty or singleton posets cannot have a full embedding of $N$ as they have fewer than 4 elements, so it suffices to show that this property is preserved by sums and joins. If $P$ and $Q$ have no full embeddings of $N$, then as $N$ is connected any full embedding $N \to P \sqcup Q$ factors through either $P$ or $Q$, and hence cannot exist. Any full embedding $N \to P \Join Q$ which does not factor through $P$ or $Q$ must send some elements to $P$ and some to $Q$. But for the embedding to be full there must be an arrow in $N$ from each element sent to $P$ to each element sent to $Q$, and there is no partition of $N$ with this property.

Now assume $P$ is a finite poset with no full embedding of $N$; we want to show that $P$ is sum-join expressible. Using strong induction on the cardinality of $P$, to show $P$ is sum-join expressible it suffices to check that $P$ is either empty, singleton, a sum of nonempty posets, or a join of nonempty posets. If $P$ is empty, singleton, or not connected, we are done, so assume $P$ is connected with more than one element. Furthermore, if $P$ has only 2 or 3 elements, it is straightforward to check by enumeration that $P$ is, up to isomorphism, sum-join expressible by one of the sum-join expressions in \eqref{eqn.smallposets}.
\begin{equation}\label{eqn.smallposets}
a \otimes b \qquad a \;\tri\; b \qquad a \otimes b \otimes c \qquad a \otimes (b \;\tri\; c) \qquad a \;\tri\; (b \otimes c) \qquad (a \otimes b) \;\tri\; c \qquad a \;\tri\; b \;\tri\; c
\end{equation}

We can therefore assume that $P$ has at least 4 elements. Let $P_{max}$ denote the maximal elements of $P$, $P_\bot$ denote the full sub-poset of $P$ containing the elements which are strictly less than every element in $P_{max}$, and $P_\top$ denote the full sub-poset of $P$ on the complement of $P_\bot$. We show that $P_\bot$ and $P_\top$ are nonempty and $P = P_\bot \Join P_\top$, completing the proof. 

$P_\top$ is nonempty as it includes $P_{max}$ which is always nonempty for finite posets. If $|P_{max}| \le 2$ then as $P$ is connected the maximal elements must have a common predecessor, so $P_\bot$ is nonempty. If $|P_{max}| > 2$, note that as $P$ is connected there exists a minimal-sized (and therefore pairwise-incomparable) subset $P_{connect}$ of elements in $P \backslash P_{max}$ such that every pair of elements in $P_{max}$ is related by a zigzag of comparisons of the form $a < b$ with $a \in P_{connect}$ and $b \in P_{max}$. If $P_{connect}$ contains more than one element, then there must exist $b,d \in P_{max}$ and $a,c \in P_{connect}$ which form an embedded copy of $N$ in $P$: take $a,c$ to be any pair which precede overlapping subsets of $P_{max}$ (such a pair must exist as $P$ is connected), and let $b$ be a maximal element with $a < b > c$. If the maximal elements preceded by $c$ were all also preceded by $a$, $P_{connect}$ would not be minimal as $c$ could be removed, so there exists $d \in P_{max}$ incomparable with $a$ such that $c < d$. As $P$ is assumed to have no full embeddings of $N$, $P_{connect}$ must contain just a single element preceding all of $P_{max}$, so $P_\bot$ is nonempty.

To show that $P = P_\bot \Join P_\top$, it suffices to prove that for each $c \in P_\bot$ and $a \in P_\top$, we have $c < a$ in $P$. As $P$ is connected, $a$ must be less or equal to than some maximal element $b$. If $a=b$ then as $c \in P_\bot$ we have $c < a$ and we are done, so assume that $a < b$. As $a$ is not in $P_\bot$, there must be some maximal element $d$ such that $a$ and $d$ are incomparable. We then have $a < b$, $c < b$, $c < d$, and $d$ separated from $a$ and $b$ (as all maximal elements are separated). If we do not have $c < a$, then the elements $a,b,c,d$ inherit from $P$ precisely the partial order $N$ from \eqref{eqn.zigzag}. Therefore as $P$ contains no full embeddings of $N$, we must have $c < a$.  
\end{proof}

\section{Equivalence with duoidal structure maps}

To show that sum-join expressible posets and identity-on-objects inclusions between them correspond to free sum-join expressions and structure maps between them in any physical duoidal category $\cat{C}$, we show that for two fixed sum-join expressions, the identity-on-objects inclusions between the corresponding posets are in bijection with the structure maps between the corresponding objects in $\cat{C}$.  This implies both that each poset is represented by at most one sum-join expression, and the algebraic structure of physical duoidal categories can be encoded entirely in terms of sum-join expressible posets. The manner of this encoding using categorical operads is discussed in \cref{sec.operads}. 

\begin{definition}
By \emph{physical duoidal expression} we will denote a well-formed term in the language consisting of a nullary symbol $\yon$ and two binary symbols $\otimes,\tri$, in which each variable appears only once. Two physical duoidal expressions are \emph{equivalent} if they are related by some combination of associativity of $\otimes$ and $\tri$, unitality of $\yon$ with respect to $\otimes,\tri$, and symmetry of $\otimes$.
\end{definition}

The equivalence classes of physical duoidal expressions in $n$ fixed variables can be identified with the set $\Cat{O}_n$ of $n$-ary operations in the symmetric operad $\Cat{O}$ defined by the pushout
\[
\begin{tikzcd}
	\Cat{P}\ar[r, hook]\ar[d, hook]&\Cat{CM}\ar[d]\\
	\Cat{M}\ar[r]&\Cat{O}\ar[ul, phantom, very near start, "\ulcorner"]
\end{tikzcd}
\] 
Here $\Cat{P}$ is the operad for pointed sets, with just the identity and a single nullary operation, $\Cat{M}$ is the symmetric operad for monoids where $\Cat{M}_n$ has a single $n$-ary operation for each permutation of $n$ variables, and $\Cat{CM}$ is the symmetric operad for commutative monoids where $\Cat{CM}_n$ has a single $n$-ary operation invariant under permutation of the variables. The associativity of $\otimes$ and $\tri$ in physical duoidal expressions is encoded by $\Cat{CM}$ and $\Cat{M}$, and the shared unit is encoded by identifying the nullary operations in the two copies via $\Cat{P}$. 

Given a physical duoidal expression $p$ with $n$ variables and additional physical duoidal expressions $p_1,...,p_n$ with mutually distinct variables, there is a physical duoidal expression denoted $p \circ (p_1,...,p_n)$ obtained by substituting $p_i$ into the $i$th variable of $p$ for $i=1,...,n$. This is precisely the composition operation in the operad $\Cat{O}$.

\begin{definition}[Duoidal structure maps]
Morphisms between physical duoidal expressions with the same variables---which we call \emph{duoidal structure maps}---are inductively generated by the lax interchanger
\begin{equation}\label{eqn.interchanger1}
(a \;\tri\; b) \otimes (c \;\tri\; d) \to (a \otimes c) \;\tri\; (b \otimes d) 
\end{equation}
under the following operations:
\begin{itemize}
	\item Equivalence: any two equivalent physical duoidal expressions have morphisms between them in both directions. This in particular includes the identity morphism from any physical duoidal expression to itself
	\item Composition: given morphisms of physical duoidal expressions $p \to q$ and $q \to r$, there is a morphism $p \to r$;
	\item Products: given morphisms of physical duoidal expressions $p \to q$ and $p' \to q'$, there are morphisms $p \otimes p' \to q \otimes q'$ and $p \;\tri\; p' \to q \;\tri\; q'$;
	\item Substitution: given a morphism $p \to q$ between physical duoidal expressions with $n$ variables, and morphisms $p_1 \to q_1$, ..., $p_n \to q_n$, there is a morphism $p \circ (p_1,...,p_n) \to q \circ (q_1,...,q_n)$;
\end{itemize} 
along with equations ensuring that any two morphisms between the same fixed physical duoidal expressions are equal.
\end{definition}

We now proceed to show that these morphisms are precisely the same as the identity-on-elements inclusions between the corresponding sum-join expressible posets. To do so, we need a notion of substitution for posets as well.

\begin{definition}[Lexicographic substitution]\label{substitution}
Given a finite poset $P$ and for each $a \in P$ a poset $P_a$, the substitution poset denoted $P \circ (P_a)_{a \in P}$ has elements $\coprod_{a \in P} P_a$ such that $x \le y$ for $x \in P_a$ and $y \in P_b$ when either $a < b$ in $P$ or $a=b$ and $x \le y$ in $P_a$. This poset structure is called the \emph{lexicographic} order.
\end{definition}

Suppose $p$ is a physical duoidal expression with variables $a_1,\ldots,a_n$ and that it corresponds to the $n$-element poset $P$. It is straightforward to check that if physical duoidal expressions $p_1,...,p_n$ correspond to the posets $P_{a_1},...,P_{a_n}$, then $p \circ (p_1,...,p_n)$ corresponds to $P \circ (P_{a_1},...,P_{a_n})$.

\begin{theorem}\label{fullyfaithful}
For two fixed physical duoidal expressions $p,q$ in the same $n$ variables, there is a unique duoidal structure map from $p$ to $q$ if and only if there is an identity-on-elements inclusion from the poset $P$ expressed by $p$ to the poset $Q$ expressed by $q$. 
\end{theorem}

\begin{proof}
For the ``only if'' direction, we first observe that any equivalence corresponds to the identity morphism of posets and the lax interchanger \eqref{eqn.interchanger1} corresponds to the inclusion of posets in \eqref{eqn.interchanger2}.
\begin{equation}\label{eqn.interchanger2}
\left(\begin{tikzcd}
b & d \\
a \uar & c \uar
\end{tikzcd}\right)\quad\emb\quad\left(\begin{tikzcd}
b & d \\
a \uar \ar{ur} & c \uar \ar{ul}
\end{tikzcd}\right)
\end{equation}
It then suffices to show that identity-on-elements inclusions of posets are preserved by $\sqcup$, $\Join$, composition, and substitution. The first three are straightforward because inclusions of posets are closed under composition and because $\sqcup,\Join$ are functors that extend $\sqcup$ on the underlying sets. Hence it remains only to check that substitution is functorial. 

Consider an identity-on-elements inclusion $P \to Q$ between sum-join expressible posets with elements $a_1,...,a_n$, and identity-on-elements inclusions $P_{a_1} \to Q_{a_1}$, ..., $P_{a_n} \to Q_{a_n}$. It suffices to show that the identity function is order-preserving from $P \circ (P_{a_1},...,P_{a_n})$ to $Q \circ (Q_{a_1},...,Q_{a_n})$. Let $x \le y$ in $P \circ (P_{a_1},...,P_{a_n})$ for $x \in P_{a_i}$ and $y \in P_{a_j}$. If $a_i < a_j$ in $P$ then $a_i < a_j$ in $Q$, while if $a_i=a_j$ and $x \le y$ in $P_{a_i}$ then $x \le y$ in $Q_{a_i}$, so either way $x \le y$ in $Q \circ (Q_{a_1},...,Q_{a_n})$.

For the ``if'' direction, we first observe as a base case that the identity morphisms on the empty and singleton posets correspond to the identity structure maps on $\yon$ and the single-variable expression.

We now proceed by strong induction, noting that if $P$ and $Q$ each have the same $n$ elements for $n > 1$ then each is either a sum or a join of smaller nonempty posets. We check in each of the resulting cases that if the identity function is a poset inclusion then it corresponds to a duoidal structure map from $p$ to $q$. 

If $Q$ is a sum of $Q_1$ and $Q_2$, and the identity function is a poset inclusion from $P$ to $Q$, then the elements of $Q_1$ and $Q_2$ must also be mutually disjoint in $P$ (which therefore cannot be a join). The identity inclusion is then the sum of two identity-on-elements inclusions on strictly smaller posets, which we inductively assume to correspond to duoidal structure maps. The identity inclusion from $P$ to $Q$ then corresponds to applying $\otimes$ to these two structure maps. 

If $P$ is a join of $P_1$ and $P_2$, and the identity is a poset inclusion to $Q$, then the elements of $P_1$ must also relate to all the elements of $P_2$ in $Q$ (which therefore cannot be a disjoint union). The identity inclusion is then the join of two identity-on-elements inclusions on strictly smaller posets, so we have similarly inductively exhibited this identity-on-elements-inclusion as corresponding to applying $\tri$ to two duoidal structure maps.

Finally, assume $P$ is a sum of $P_1$ and $P_2$ and that $Q$ is a join of $Q_1$ and $Q_2$, and that the identity is a poset inclusion from $P$ to $Q$. Consider the posets $P_{i,j}$ where $i$ and $j$ range over $1,2$ and $P_{i,j}$ is the full sub-poset of $P_i$ on the vertices which overlap with $Q_j$. Similarly define $Q_{i,j}$ as the full sub-poset of $Q_j$ on the vertices which overlap with $P_i$. By \cref{expressibility} $P_{i,j}$ and $Q_{i,j}$ are sum-join expressible: indeed, each is a full sub-poset of a poset with no full embeddings of $N$, so neither can have a full embedding of $N$. The identity inclusion from $P$ to $Q$ factors as in \eqref{eqn.factorization}.
\begin{equation}\label{eqn.factorization}
P = P_1 \sqcup P_2 \emb (P_{1,1} \Join P_{1,2}) \sqcup (P_{2,1} \Join P_{2,2}) \emb (Q_{1,1} \sqcup Q_{2,1}) \Join (Q_{1,2} \sqcup Q_{2,2}) \emb Q_1 \Join Q_2 = Q
\end{equation}
Indeed, the identity function is a poset inclusion from $P$ to $Q$, the restricted identity functions are poset inclusions from $P_{i,j}$ to $Q_{i,j}$, $P_i$ to $P_{i,1} \Join P_{i,2}$, and $Q_{1,j} \sqcup Q_{2,j}$ to $Q_j$. As each of these posets has fewer elements than $P$ and $Q$, we can inductively assume that each corresponds to a duoidal structure map.

This shows immediately that the first and third inclusions in \eqref{eqn.factorization} correspond to duoidal structure maps, so it suffices to show the same for the second inclusion. But this inclusion can be recovered by substituting
\[ P_{1,1} \emb Q_{1,1}, \qquad P_{1,2} \emb Q_{1,2}, \qquad P_{2,1} \emb Q_{2,1}, \qquad P_{2,2} \emb Q_{2,2} \]
respectively into the elements $a,b,c,d$ in the lax interchanger \eqref{eqn.interchanger1}, which completes the proof as we have thus exhibited $P \emb Q$ as a composite of inclusions which correspond to duoidal structure maps.
\end{proof}

We have now constructed a faithful functor from the category of physical duoidal expressions and morphisms between them to the category of posets. This means that the category of physical duoidal expressions is equivalent to the image of that functor, namely the category of sum-join expressible posets and identity-on-elements inclusions between them. Another way of phrasing this is that sum-join expressible posets and identity-on-elements inclusions are a free physical duoidal category, and we next explore an operadic formalism for representing this.

\section{Formalization using categorical operads}\label{sec.operads}

Just as certain algebraic structures on sets can be described using operads---sequences of sets $O_n$ for $n \in \nn$ of $n$-ary operations equipped with identity, composite, and sometimes symmetry operations---so too can algebraic structures on categories be described using categorical operads, where the sets $O_n$ are replaced by categories $\O_n$ and the identity and composition functions replaced by functors.\footnote{Categorical operads are also sometimes known as $\smcat$-enriched operads.} The category structure on the algebraic operations permits the encoding of coherence maps between different operations, such as the lax interchanger between two different 4-ary operations. We say that $\cat{O}$ is the categorical operad for some 2-category $D$ if $D$ is equivalent to the 2-category of pseudo-algebras for $\cat{O}$.

While most popular algebraic structures on categories are indeed modeled by categorical operads, these operads are often defined in the same way as the algebraic structures themselves: for instance, the categorical operad describing duoidal categories is generated under operadic unit and composition by the binary and nullary operations $\otimes,\tri,\yon_\otimes,\yon_\tri$ and the morphisms between them from \eqref{eqn.duoidal}, with equations guaranteeing any two morphisms between the same operations are equal.

This ``presentation'' style of definition is of course common across mathematics and technically sound, but it does not readily provide a description of the categories $\O_n$ which concretely represent the compound operations and coherences a categorical algebraic structure contains. It is therefore helpful for the theory of a particular algebraic structure when its corresponding operad can be represented concretely in terms of familiar mathematical objects. 

\begin{remark}
For instance, in homotopy theory there is a notion of a homotopy-coherent monoid (or $A_\infty$-algebra) which satisfies the unit and associativity equations up to homotopies, which satisfy further coherence axioms up to higher homotopies, and so on. There are well-understood patterns to these coherences represented in terms of associahedra polytopes, but in practice it is difficult to work with a higher operad defined only by these generators and their relations. Instead, concrete models such as the equivalent ``little intervals'' operad provide a more practical alternative in terms of distinct but well-understood mathematical objects.
\end{remark}

For physical duoidal categories, this concrete model is achieved using sum-join expressible posets. We now describe a concrete categorical operad whose pseudo-algebras are precisely physical duoidal categories, using posets.

\begin{definition}[Categorical symmetric operad of finite posets]
Define the categorical symmetric operad of \emph{finite posets}, denoted $\FP$, as follows:
\begin{itemize}
	\item its category $\FP_n$ of $n$-ary operations is the category of poset structures on the set $\ul n$ and identity-on-elements inclusions between them (in fact $\FP_n$ is itself a poset);
	\item its unit functor $\eta\colon 1 \to \FP_1$ is an isomorphism, since there is only one poset structure on $\ul 1$;
	\item its composition functor 
\[
\mu\colon \FP_n \times \FP_{m_1} \times \cdots \times \FP_{m_n} \to \FP_{m_1 + \cdots + m_n}
\]
sends $(P,P_1,...,P_n)$ to the \emph{lexicographic} poset structure on 
\[
\ul{m_1 + \cdots  + m_n} \cong \{(1,1),...,(1,m_1),...,(n,1),...,(n,m_n)\}
\]
i.e.\ the one for which $(i,j) \le (i',j')$ when either $i < i'$ in $P$ or when $i=i'$ and $j \le j'$ in $P_i$ (this agrees with the substitution operation in \cref{substitution});
	\item for $\tau \in \Sigma_n$ a permutation on $\ul n$, the symmetry isomorphism $\sigma_\tau\colon \FP_n \To{\cong} \FP_n$ is given by applying $\tau$ to the set $\ul n$ underlying the posets in $\FP_n$.
\qedhere
\end{itemize}
\end{definition}

%
In $\FP$, sum-join expressible posets are closed under identities and operadic composition. This is because the identity is the one element poset, which is sum-join expressible, and each sum-join expressible poset is a composition of joins and sums. The sum of posets $P$ and $Q$ can be realized as the operadic composite of the two element discrete poset with $(P,Q)$, while the join of $P$ and $Q$ is the operadic composite of the poset $1 < 2$ with $(P,Q)$. Therefore, the composite of any sum-join expressible poset $P$ on $\ul n$ with posets $(P_1,...,P_n)$ is given by an $n$-ary expression of sums (disjoint unions) and/or joins applied to $P_1,...,P_n$, which of course preserves sum-join expressibility.

\begin{definition}
We denote by $\Expr$ the full sub-operad of $\FP$ consisting of the sum-join expressible posets.
\end{definition}

Before stating the main theorem, we first recall the definition of a pseudo-algebra for a categorical operad.

\begin{definition}[Pseudoalgebra]
For a categorical symmetric operad $\O$, an $\O$-pseudoalgebra is a category $\C$ equipped with structure maps
\[
\boxtimes_n\colon \O_n \times \C^n \to \C
\]
and natural isomorphisms as in \eqref{eqn.algebra}
\begin{equation}\label{eqn.algebra}
\begin{tikzcd}[column sep=100]
\O_n \times \O_{m_1} \times \cdots \times \O_{m_n} \times \C^{m_1} \times \cdots \times \C^{m_n} \rar{\O_n \times \boxtimes_{m_1} \times \cdots \times \boxtimes_{m_n}} \dar[swap]{\mu \times \C^{m_1+\cdots+m_n}} & \O_n \times \C^n \dar{\boxtimes_n} \\
\O_{m_1+\cdots+m_n} \times \C^{m_1+\cdots+m_n} \rar[swap]{\boxtimes_{m_1+\cdots+m_n}} \ar[phantom,shift right=1]{ur}{\cong} & \C
\end{tikzcd}
\end{equation}
\begin{equation*}\begin{tikzcd}[column sep=large]
\C \rar{\eta \times \C} \ar[equals]{dr} & \O_1 \times \C \dar{\boxtimes_1} \dar[phantom, shift right=6, near start]{\cong} \\
& \C
\end{tikzcd}\qquad\qquad\qquad\begin{tikzcd}[column sep=60]
\O_n \times \C^n \rar{\O_n \times \C^\tau} \dar[swap]{\sigma_\tau \times \C^n} & \O_n \times \C^n \dar{\boxtimes_n} \\
\O_n \times \C^n \rar[swap]{\boxtimes_n} \ar[phantom]{ur}{\cong} & \C
\end{tikzcd}
\end{equation*}
satisfying unit, associativity, and equivariance equations. 
\end{definition}

For a pseudoalgebra $\cat{C}$ and an $n$-element poset $P\in\FP_n$, we will let $\boxtimes_n^P\colon \C^n \to \C$ denote the partial application $\boxtimes_n^P\coloneqq\boxtimes_n(P,-)$. For example, given a poset such as
\[
\begin{tikzcd}
b & d \\
a \uar & c \uar \ar{ul} 
\end{tikzcd}
\]
we get a functor $\cat{C}^4\to\cat{C}$. A pseudoalgebra $\C$ also has natural morphisms between these operations for all maps between the relevant posets, and coherence isomorphisms for iterated, null, and permuted applications of these operations. 

\begin{theorem}\label{expr_operad}
Pseudoalgebras for $\Expr$ are precisely the physical duoidal categories.
\end{theorem}

\begin{proof}
To show that an $\Expr$-pseudoalgebra $\C$ forms a physical duoidal category, we define the unit $\yon\in\cat{C}$ to be $\boxtimes_0^\varnothing \colon 1 \to \C$, and we define $\otimes,\tri$ by $\boxtimes_2^{(1 \;\; 2)}, \boxtimes_2^{(1 < 2)} \colon \C^2 \to \C$ respectively. Restricting $\Expr$ to the discrete full sub-operads generated by $(1 \;\; 2)$ and $(1 < 2)$, namely the operads $\Cat{CM}$ and $\Cat{M}$ for commutative monoids and monoids, shows that $(\yon,\otimes)$ and $(\yon,\tri)$ form a symmetric monoidal and monoidal structure on $\C$ as these are precisely the pseudoalgebras for $\Cat{CM}$ and $\Cat{M}$. 

The category $\Expr_n$ is equivalent by \cref{fullyfaithful} to the category of physical duoidal expressions on $n$ fixed variables, which ensures that all of the physical duoidal structure maps are present in $\C$ between the $n$-ary functors $\boxtimes_n^P$. But these agree with the appropriate composites of $\yon,\otimes,\tri$ up to coherent natural isomorphism by the pseudoalgebra structure, and so $\C$ forms a physical duoidal category.

Conversely, given a physical duoidal category structure on $\C$, for a fixed choice of variables $a_1,...,a_n$ choose a representative $p$ of each equivalence class of physical duoidal expressions on those $n$ variables. Then for $P$ the corresponding sum-join expressible poset, define $\boxtimes_n^P \colon \C^n \to \C$ by the formula $p$. The choice of $p$ for each sum-join expressible poset $P$ amounts to a choice of quasi-inverse functor to the equivalence of categories from physical duoidal expressions on $n$ variables to $\Expr_n$, so these assignments are functorial in $P$. The coherence isomorphisms in \eqref{eqn.algebra} are then uniquely derived from the coherences of the normal duoidal structure.

It is straightforward to check that these constructions are inverse to one another up to isomorphism, completing the proof.
\end{proof}

Equivalently, this shows that the category $\coprod\limits_n \Expr_n$ with symmetries added in, namely the category of finite sum-join expressible posets and bijective-on-elements maps of posets, is the free (weakly symmetric, strictly unital and associative) physical duoidal category generated by one object.

\chapter{Dependence Categories}\label{chap.dependence}

Having shown in \cref{expr_operad} that the operad for physical duoidal categories is a full sub-operad of $\FP$, it is natural to wonder what the pseudoalgebras for $\FP$ look like. These are categories with, in addition to a physical duoidal structure, operations $\boxtimes_n^P \colon \C^n \to \C$ for all posets $P$ on $\ul n$. 

\section{Definition and first examples}

\begin{definition}
A \emph{dependence category} is an $\FP$-pseudoalgebra.
\end{definition}

Intuitively, this allows for objects to be juxtaposed according to more complicated causal structures than simply spatially and temporally. Many of our examples of physical duoidal structures extend naturally to dependence structures.

\begin{example}\label{posetdependence}
Definitionally, the category $\coprod\limits_n \FP_n$ with symmetries added in, namely the category of finite posets and bijective-on-elements maps of posets, forms an $\FP$-pseudoalgebra. But as the composition product $P \circ (P_1,...,P_n)$ of posets is in fact functorial in $P_1,...,P_n$ with respect to all maps of posets, this dependence structure extends to the entire category of posets.
\end{example}

\begin{example}
$\mink$ is not a dependence subcategory of posets, as for a finite poset $P$ that does not embed into Minkowski space $\boxtimes_n^P(\cdot,...,\cdot) = P$ is not in $\mink$ even though the singleton poset is. However, posets embeddable in Minkowski space and identity-on-elements inclusions between them form a full sub-operad of $\FP$ containing $\Expr$. 

To see this, consider such posets $P,P_1,...,P_n$. Any embedding into Minkowski space is scalar-invariant, and can be perturbed so that every point in the embedding dependent on $p$ is inside the light cone of $p$ rather than on its boundary. To embed $P \circ (P_1,...,P_n)$ then, take such a perturbed embedding of $P$ and replace each point $p_i$ with an embedding of $P_i$ centered at $p_i$ and scaled down to be so small as to be included in the intersection of all forward and backward light cones that $p_i$ belongs to. Each $P_i$ is then embedded in such a way that preserved its own dependencies and also inherits those of $i \in P$. 

The same construction can be described similarly using sphere arrangements, where the down-scaling of embeddings in $\rr^4$ corresponds to enlarging the spheres in an arrangement representing $P_i$ so much that their joint interior closely resembles the interior of a single sphere, where this now-thin arrangement replaces the $i$th sphere in the representation of $P$.
\end{example}

\begin{example}\label{tropicaldependence}
As we will see in \cref{sec.dep_phys}, most of our examples of dependence categories can be derived from a physical duoidal structure using limits. However, $\rrnn$ carries a dependence structure which we can define directly. The physical duoidal fragment agrees with \cref{tropical}, where $\otimes$ is given by $\max$ and $\tri$ is given by $+$, and the more general operations are again motivated by the runtimes of parallel programs. 

Given $n$ programs, a poset $P$ on $\ul n$ can be interpreted as describing dependencies between the programs, where $i < j$ in $P$ if the $j$th program requires the output of the $i$th program before it can begin running. Given runtimes $a_1,...,a_n \in \rrnn$ of these programs, we define $\boxtimes_n^P(a_1,...,a_n)$ as the minimal amount of time it would take to run these programs given unlimited parallel computing resources, e.g.\ with access to $n$ machines, such that whenever $i < j$ in $P$, the $i$th program is completed before the $j$th program begins to run. Define
\[ \boxtimes_n^P(a_1,...,a_n) = \max_{i_1<\cdots<i_k} \; \sum_{\ell = 1}^k \; a_{i_\ell}. \]

This is the minimal possible runtime of such an arrangement of programs, as for any increasing sequence $i_1<\cdots<i_k$ in $P$, the $i_1$th program must complete running before the $i_2$th program begins and so on, and this runtime is achievable by first running in parallel the programs associated to all minimal elements of $P$ and then beginning each subsequent program as soon as all of its prerequisites in $P$ have completed. This is related to the \emph{critical path method} of event scheduling, where the sequence $i_1 < \cdots < i_k$ with the longest total runtime is called the critical path (see \cite{criticalpath}).

To see that this is a dependence category, first observe that for any identity-on-elements inclusion $P \to Q$, $\boxtimes_n^P(a_1,...,a_n) \le \boxtimes_n^Q(a_1,...,a_n)$ as any increasing sequence in $P$ is also increasing in $Q$ and $\max$ is monotone with respect to inclusions on indexing sets. This assignment respects units as $\boxtimes_1^\cdot(a) = a$, and respects composites by the distributivity of $\sum$ over $\max$. In particular, given posets $P,P_1,...,P_n$ on $\ul n,\ul m_1,...,\ul m_n$ and $a_{1,1},...,a_{n,m_n} \in \rrnn$, any increasing sequence through $P \circ (P_1,...,P_n)$ with the longest total runtime is given by finding the corresponding sequence for $P$ and plugging into each of its elements $i$ the analogous sequences for $P_i$. 
\end{example}

\begin{example}
Any symmetric monoidal category has a dependence structure where for fixed $n$, each operation $\boxtimes_n^P$ is given by the $n$-ary tensor product. All of the necessary structure maps are then isomorphisms. Indeed, they are given by identities or compositions of symmetry, associativity, and unit isomorphisms.
\end{example}

\section{Dependence Categories from Physical Duoidal Categories}\label{sec.dep_phys}

In most of our examples of dependence categories, such as \cref{posetdependence}, we can begin with a physical duoidal category and then observe that the operation associated to an arbitrary poset $P \in \FP_n$ can be extracted from the $n$-ary $\tri$ operation as the ``subobject of tuples in which only the dependencies in $P$ are allowed.'' We can encode this process more formally as a procedure for recovering the dependencies of any $P$ from other posets which are sum-join expressible.

\begin{example}
Recall the poset $N$ from \eqref{eqn.zigzag}.
\[ 
N = \left(\begin{tikzcd}
b & d \\
a \uar & c \uar \ar{ul} 
\end{tikzcd}\right)
\]
While $N$ is not sum-join expressible, it is the pullback (namely, intersection) of the cospan of sum-join expressible posets in \eqref{eqn.intersect}.
\begin{equation}\label{eqn.intersect}
\left(\begin{tikzcd}
b & d \\
a \uar \ar{ur} & c \uar \ar{ul} 
\end{tikzcd}\right)\quad\hookrightarrow\quad\left(\begin{tikzcd}
b & d \\
a \uar \ar{ur} & c \lar \uar \ar{ul} 
\end{tikzcd}\right)\quad\hookleftarrow\quad\left(\begin{tikzcd}
b & d \\
a \uar & c \lar \uar \ar{ul} 
\end{tikzcd}\right)
\end{equation}
These posets are expressed respectively by the physical duoidal expressions
\[ \; (a \otimes c) \;\tri\; (b \otimes d) \;\;\;\quad\qquad c \;\tri\; a \;\tri\; (b \otimes d) \qquad\quad\;\;\; c \;\tri\; ((a \;\tri\; b) \otimes d). \]
Furthermore, while this cospan suffices to recover $N$ by intersection, it also shows that $N$ is the intersection of all sum-join expressible poset structures on $a,b,c,d$ containing $N$: every ordered pair of elements in $N$ which are not related in $N$ are also not related in some sum-join expressible poset containing $N$. 
\end{example}

This example illustrates a more general strategy: even a non-sum-join-expressible poset $P$ in $\FP_n$ arises as the limit of the sum-join expressible posets containing it, by exhibiting $P$ as their intersection.

\begin{lemma}\label{lem.expr_connected}
For any poset $P$ on $\ul n$, the category $P/\expr$ is connected.
\end{lemma}

\begin{proof}
We first claim that if $P/\expr$ is the union of full subcategories of the form $Q(D)/\expr$ for a connected fully faithful functor $Q \colon \mathcal{D} \to P/\FP_n$ such that each $Q(D)/\expr$ is connected, then $P/\expr$ is connected. To see this, observe that for $Q(D) \hookrightarrow Q(D')$ in $P/\FP_n$ there is a full subcategory inclusion $$Q(D)/\expr \hookleftarrow Q(D')/\expr.$$ Therefore if $Q(D)/\expr$ is connected and there is a cospan $$Q(D'') \hookleftarrow Q(D) \hookrightarrow Q(D'),$$ any object in $Q(D'')/\expr$ has a zigzag in $Q(D)/\expr$ (and hence in $P/\expr$) to any object in $Q(D')/\expr$. The same then applies for any $Q(D'')$ and $Q(D')$ related by a zigzag in $P/\FP_n$, and by assumption any two objects in $P/\expr$ belong respectively to some subcategories $Q(D'')/\expr$ and $Q(D')/\expr$ related by such a zigzag.

We also note that the category $N/\Cat{Expr}_4$ is connected, having the form in \eqref{eqn.nexpr}. 
\begin{equation}\label{eqn.nexpr}
\begin{tikzcd}[row sep=small,column sep=-35]
&[45] &[37] (a \otimes (c \;\tri\; d)) \;\tri\; b \ar{dr} \ar{ddl} & & a \;\tri\; c \;\tri\; d \;\tri\; b \\
& & & (a \otimes c) \;\tri\; d \;\tri\; b \ar{dl} \ar{ur} & & a \;\tri\; c \;\tri\; (b \otimes d) \ar{dr} \ar{ul} \\
c \;\tri\; d \;\tri\; a \;\tri\; b & c \;\tri\; (a \otimes d) \;\tri\; b \lar \rar & c \;\tri\; a \;\tri\; d \;\tri\; b & & (a \otimes c) \;\tri\; (b \otimes d) \ar{dr} \ar{dl} \ar{ur} \ar{ul} & & a \;\tri\; c \;\tri\; b \;\tri\; d \\
& & & c \;\tri\; a \;\tri\; (b \otimes d) \ar{dr} \ar{ul} & & (a \otimes c) \;\tri\; b \;\tri\; d \ar{dl} \ar{ur} \\
& & c \;\tri\; ((a \;\tri\; b) \otimes d) \ar{ur} \ar{uul} & & c \;\tri\; a \;\tri\; b \;\tri\; d
\end{tikzcd}
\end{equation}
This can be checked by considering every poset $P$ on the vertices $a,b,c,d$ containing the relations in $N$, observing that every such $P\supsetneq N$ is sum-join expressible and that these $P$'s are precisely the posets in \eqref{eqn.nexpr}.

We can now proceed by strong induction on $n-h(P)$, where $h(P)$ is the height of the poset $P$.\footnote{Recall that the height is the natural number $h$ such that the largest possible ordinal with an injection into $P$ has the form $0 < \cdots < h$.} In the base case of $n-h(P)=0$, we have that $P$ is of height $n$ and has $n$-many vertices, so it is a linear order. Hence, $P/\expr$ is the terminal category, which is connected. Now assume the result holds for posets of height at least $h(P)=H+1$ and assume $h(P)=H$.

If $P$ is expressible, we are done. Otherwise, there is some fully embedded copy of $N$ in $P$; we will denote the vertices of its image in $P$ by $a,b,c,d$. For any poset in $P/\expr$, the full sub-poset on the vertices $a,b,c,d$ must have one of the forms in \eqref{eqn.nexpr}. Therefore $P/\expr$ is the union of the full subcategories $Q(D)/\expr$ where $Q \colon N/\Cat{Expr}_4 \to P/\FP_n$ and $Q(D)$ is the poset on $\ul n$ generated by the union of the relations in $P$ and those on $a,b,c,d$ coming from the poset $D$. As $N/\Cat{Expr}_4$ is connected, it suffices to show that each $Q(D)/\expr$ is connected.

Iterating the argument above with $Q(D)$ in place of $P$ results in a tree of posets on $\ul n$ with root $P$ such that for each node $P'$, if $P'$ is expressible or of height greater than $H$ it is a leaf, and otherwise its children are the posets $Q'(D)$ for some functor $Q' \colon N/\Cat{Expr}_4 \to P'/\FP_n$ of the form described for $P$ above. Since $Q'(D)$ has strictly more relations than $P'$, the height of a poset on $\ul n$ is non-decreasing with respect to the number of relations in the poset, and a poset on $\ul n$ with the maximal number of relations also has maximal height ($n$), this tree must be finite as every path down from the root will eventually reach either an expressible poset or a poset of height greater than $H$. 

Our result holds for each leaf in this tree either by the inductive hypothesis or the initiality of an expressible poset $P'$ in $P'/\expr$, and by the argument above if the result holds for every child of a node it holds for the node as well. Therefore the result must hold for the root of the tree, $P$, completing the inductive step of the proof.
\end{proof}

\begin{remark}
Curiously, the limit in $\FP_4$ of the diagram in \eqref{eqn.nexpr} is also $N$, despite it appearing to be a very different diagram from the cospan in \eqref{eqn.intersect}. One way to see this is by noting that limits in $\FP_n$ are intersections of posets, and at least one of the posets $(a \otimes c) \;\tri\; (b \otimes d)$ and $c \;\tri\; ((a \;\tri\; b) \otimes d)$ includes into every poset in \eqref{eqn.nexpr} except for $(a \otimes (c \;\tri\; d)) \;\tri\; b$, which is itself the intersection of other posets in the diagram. This ensures that every pair of vertices which is incomparable in some poset in \eqref{eqn.nexpr} is also incomparable in the limit of \eqref{eqn.intersect}, hence their agreement.
\end{remark}

\begin{proposition}\label{expressible_limit}
$\FP_n$ is generated under connected limits by its full subcategory of sum-join expressible posets.
\end{proposition}

In other words, every poset structure in $\FP_n$ is a connected limit of sum-join expressible ones, and every inclusion is induced by morphisms of connected diagrams of sum-join expressible poset structures.

\begin{proof}
We first show that each $P\in\FP_n$ is the limit of the inclusion functor from $P/\expr$ to $\FP_n$, where $P/\expr$ is the full subcategory of the undercategory $P/\FP_n$ spanned by the sum-join expressible poset structures, with the evident inclusion into $\FP_n$. This will complete the result on objects, by Lemma~\ref{lem.expr_connected}.

As a meet-complete poset, $\FP_n$ has limits computed as intersections of order relations. As in the undercategory of $P$ every poset on $\ul n$ contains all of the inequalities from $P$, it suffices to show that for any elements $i,j$ incomparable in $P$, there exists a poset in $P/\expr$ within which $i,j$ are incomparable. To do this, we use the fact that any poset $P$ with elements $\ul n$ includes into a linear order on $\ul n$.\footnote{To see this, define the linear order inductively by choosing any minimal element of $P$ as the first element in the linear order, then repeating the process for the remaining elements in $P$ until all are in order.}

Fix a linear order $R$ in $P/\expr$ which we will denote as if it were the canonical order on $\ul n$, let $i,j$ be incomparable in $P$, and assume $i <_R j$. For all $k$ with $i <_R k <_R j$, we have at most one of $i <_P k$ or $k <_P j$. Let $Q_i$ be the linear order inherited from $R$ on the elements $i$ and all $k$ such that $i <_R k <_R j$ and $i <_P k$, and let $Q'_j$ be the linear order inherited from $R$ on the elements $j$ and all $k$ such that $i <_R k <_R j$ but not $i <_P k$. Then 
\[
\left\{1 < \cdots < i{-}1\right\} \;\tri\; \left(Q_i \sqcup Q'_j\right) \;\tri\; \left\{j{+}1 < \cdots < n\right\}
\]
belongs to $P/\expr$ with $i,j$ incomparable. Indeed, this is sum-join expressible as a join of sums of linear orders and $i$ and $j$ are incomparable as they are separated by a sum. All of the inequalities in $P$ are included as it is linear below $i$ and above $j$, while by the discussion above all of $Q_i$ is incomparable with all of $Q'_j$ in $P$. This completes the proof that $P$ is the limit of this diagram.

To conclude, note that for inclusions $P \to Q$ in $\FP_n$, precomposition induces a functor $Q/\expr \to P/\expr$ which commutes over $\FP_n$, so the limit structure on $Q$ together with the cone of $P$ over $Q/\expr$ induce the unique morphism $P \to Q$ in $\FP_n$.

\end{proof}

We now describe how this construction of arbitrary posets from sum-join expressible ones allows us to extend many physical duoidal categories to spacetime categories. In short, when a physical duoidal category has connected limits which are preserved by $\otimes$ and $\tri$, it extends to a dependence category by defining $\boxtimes_n^P$ as a limit of physical duoidal operations.

Recall that for a categorical operad $\O$, a pseudoalgebra for $\O$ is the same as a strong functor $\O \to \End(\C)$ for some category $\cat{C}$. Here $\End(\C)$ is the categorical operad with $\End(\C)_n = \Funn(\C^n,\C)$, unit and composition given by identity and composite functors, and a strong functor of categorical operads $\O \to \P$ consists of functors $\O_n \to \P_n$ which commute with unit and composition up to coherent natural isomorphism. By \cref{expr_operad}, extending a physical duoidal category to a dependence category amounts to finding an extension of the form in \eqref{eqn.duoidallift}.
\begin{equation}\label{eqn.duoidallift}
\begin{tikzcd}
\Expr \rar \dar[hook] & \End(\C) \\ \FP \ar[dashed]{ur}
\end{tikzcd}
\end{equation}

\begin{theorem}\label{dependenceextension}
If $\C$ is a physical duoidal category with finite connected limits which are preserved by $\otimes$ and $\tri$, then $\C$ admits the structure of a dependence category with 
\[\boxtimes_n^P := \left(\lim_{Q \in P/\Expr_n}\boxtimes_n^Q\right) \colon \C^n \to \C.\]
\end{theorem}

\begin{proof}
First note that the category $P/\Expr_n$ is connected by \cref{expressible_limit}, $\End(\C)_n$ has connected limits computed componentwise, and by \cref{expr_operad} $\boxtimes_n^Q$ is defined on $\C$ for each sum-join expressible poset $Q$ and functorially on the morphisms between them. To prove that this assignment $\FP \to \End(\C)$ provides a strong functor of categorical operads as in \eqref{eqn.duoidallift}, we first observe that it agrees with the physical duoidal structure on sum-join expressible posets including the unit, $Q/\Expr_n$ having a terminal object when $Q$ is sum-join expressible, so it remains only to show that it respects operadic composition.

We need to show that for $P \in \FP_n, P_1 \in \FP_{m_1},...,P_n \in \FP_{m_n}$ and $m=m_1 + \cdots + m_n$,
\begin{equation}\label{eqn.limitcompose}
\left(\lim_{Q \in P/\Expr_n} \boxtimes_n^Q\right) \circ \left(\lim_{Q_1 \in P_1/\Expr_{m_1}} \boxtimes_{m_1}^{Q_1},...,\lim_{Q_n \in P_n/\Expr_{m_n}} \boxtimes_{m_n}^{Q_n}\right) \cong \lim_{Q' \in P\circ(P_1,...,P_n)/\Expr_m} \boxtimes_m^{Q'}.
\end{equation}

By \cref{overequivalence} below, composition lifts to an initial functor
\[
P/\Expr_n \times P_1/\Expr_{m_1} \times \cdots \times P_n/\Expr_{m_n} \cong P\circ(P_1,...,P_n)/\Expr_m,
\]
and as $\otimes$ and $\tri$ commute with connected limits, so does $\boxtimes_n^Q$ for any sum-join expressible poset $Q$. We therefore have
\[\left(\lim_{Q \in P/\Expr_n} \boxtimes_n^Q\right) \circ \left(\lim_{Q_1 \in P_1/\Expr_{m_1}} \boxtimes_{m_1}^{Q_1},...,\lim_{Q_n \in P_n/\Expr_{m_n}} \boxtimes_{m_n}^{Q_n}\right)\]
\[\cong \left(\lim_{Q \in P/\Expr_n} \boxtimes_n^Q\right) \circ \lim_{(Q_1,...,Q_n) \in P_1/\Expr_{m_1} \times \cdots \times P_n/\Expr_{m_n}} \left(\boxtimes_{m_1}^{Q_1},...,\boxtimes_{m_n}^{Q_n}\right)\]
\[\cong \lim_{Q \in P/\Expr_n} \left(\boxtimes_n^Q \circ \lim_{(Q_1,...,Q_n) \in P_1/\Expr_{m_1} \times \cdots \times P_n/\Expr_{m_n}} \left(\boxtimes_{m_1}^{Q_1},...,\boxtimes_{m_n}^{Q_n}\right)\right)\]
\[\cong \lim_{Q \in P/\Expr_n} \;\;\lim_{(Q_1,...,Q_n) \in P_1/\Expr_{m_1} \times \cdots \times P_n/\Expr_{m_n}} \boxtimes_n^Q \circ \left(\boxtimes_{m_1}^{Q_1},...,\boxtimes_{m_n}^{Q_n}\right)\]
\[\cong \lim_{(Q,Q_1,...,Q_n) \in P/\Expr_n \times P_1/\Expr_{m_1} \times \cdots \times P_n/\Expr_{m_n}} \boxtimes_n^Q \circ \left(\boxtimes_{m_1}^{Q_1},...,\boxtimes_{m_n}^{Q_n}\right)\]
\[\cong \lim_{Q' \in P \circ (P_1,...,P_n)/\Expr_m} \boxtimes_n^{Q'},\]
where the final isomorphism comes from the initiality in \cref{overequivalence}, meaning that precomposition with this functor does not change the limit of any diagram.
\end{proof}

\begin{lemma}\label{overequivalence}
For $P,P_1,...,P_n$ as in the proof of \cref{dependenceextension}, there is an initial functor 
\[
P/\Expr_n \times P_1/\Expr_{m_1} \times \cdots \times P_n/\Expr_{m_n} \cong P\circ(P_1,...,P_n)/\Expr_m
\]
given by operadic composition.
\end{lemma}

\begin{proof}
We first observe that composition defines a functor, namely that for sum-join expressible $Q,Q_1,...,Q_n$ containing $P,P_1,...,P_n$ respectively, $Q \circ (Q_1,...,Q_n)$ contains $P \circ (P_1,...,P_n)$ essentially by definition of $\circ$. 
The composition functors $\FP_n \times \FP_{m_1} \times \cdots \times \FP_{m_n} \to \FP_m$ are easily checked to be fully faithful, which implies the same for $\Expr$ and this functor between undercategories.

To show that this functor is initial, we must then demonstrate that for every sum-join expressible poset $R$ on $\ul m$ containing $P \circ (P_1,...,P_n)$, the category $P \circ (P_1,...,P_n)/\Expr_m/R$ of sum-join expressible posets of the form $Q \circ (Q_1,...,Q_n)$ containing $P \circ (P_1,...,P_n)$ and contained in $R$ is connected. 

To do this, we show that it in fact has a terminal object. Let $R_1,...,R_n$ be the restrictions of $Q'$ to the elements of $P_1,...,P_n$ respectively, and let $R'$ be the poset on $\ul n$ in which $i \le j$ if either $i=j$ or for every $x \in R_i$ and $y \in R_j$ we have $x \le y$ in $R$. The composite $R' \circ (R_1,...,R_n)$ is in $P \circ (P_1,...,P_n)/\Expr_m/Q'$, as each $R_i$ must contain $P_i$ by definition (since $R$ contains $P \circ (P_1,...,P_n)$), if $i < j$ in $P$ then for every $x \in P_i$ and $y \in P_j$ we have $x \le y$ in $P \circ (P_1,...,P_n)$ and hence in $R$, and all of the inequalities in $R' \circ (R_1,...,R_n)$ are inherited from inequalities in $R$. 

If $Q \circ (Q_1,...,Q_n)$ is in $P \circ (P_1,...,P_n)/\Expr_m/R$, then $Q_i$ must include into $R_i$ by definition of $R_i$. Likewise $Q$ must include into $R'$, as if for $i \neq j$ there is any $x \in Q_i$ and $y \in Q_j$ such that $x \le y$ does not hold in $R' \circ (R_1,...,R_n)$, there must be some $x' \in Q_i$ and $y' \in Q_j$ such that $x' \le y'$ does not hold in $R$, so as $Q \circ (Q_1,...,Q_n)$ includes into $R$ it must be the case that $i \le j$ does not hold in $Q$. As there is then no relation in $Q \circ (Q_1,...,Q_n)$ that does not hold in $R' \circ (R_1,...,R_n)$, we have that the former includes into the latter, which is therefore terminal in $P \circ (P_1,...,P_n)/\Expr_m/R$.
\end{proof}

\begin{example}[Polynomial functors]\label{polydependence}
$\poly$ has all limits by \cite[Corollary 2.1.8]{aggregation}, and by analogous arguments to the proofs of \cite[Lemma 2.1.9, Proposition 2.1.13]{aggregation}, finite connected limits are preserved by $\otimes$ and $\tri$. Therefore by \cref{dependenceextension} $\poly$ is a dependence category. For $P$ a poset on $\ul n$ and $\ell_1 < \cdots < \ell_n$ a linear order on $\ul n$ containing $P$, the operation $\boxtimes_n^P$ takes the polynomials $p_1,...,p_n$ to the cartesian subfunctor of 
\[ p_{\ell_1} \;\tri\; \cdots \;\tri\; p_{\ell_n} = \sum_{I_1 \in p_{\ell_1}(1)} \; \prod_{i_1 \in p_{\ell_1}[I_1]} \; \cdots \; \sum_{I_n \in p_{\ell_n}(1)} \; \prod_{i_n \in p_{\ell_n}[I_n]} \; \yon \]
consisting of tuples $(I_1,...,I_n)$ in which for each $j < k$, $I_k$ is independent of $i_j$ unless $\ell_j < \ell_k$ in $P$. 
\end{example}

\section{Cocartesian Dependence Categories}\label{sec.cocart}

As discussed in \cref{cocartmonoidal}, any monoidal category with finite coproducts forms a duoidal category, and if the monoidal unit is initial, a physical duoidal category. Furthermore, if coproducts commute with finite connected limits (as they so often do), by \cref{dependenceextension} we have a dependence category. 

\begin{example}
In the category $\smset_\ast$ of pointed sets, the initial object is also terminal, so coproducts and products provide a physical duoidal structure. Coproducts (given by wedge sums) commute with connected limits, so $\smset_\ast$ forms a dependence category. For a poset $P$ on $\ul n$ and pointed sets $X_1,...,X_n$, the pointed set $\boxtimes_n^P(X_1,...,X_n)$ is the subset of their product containing tuples $(x_1,...,x_n)$ such that unless $i < j$ in $P$, at least one of $x_i$ and $x_j$ is the basepoint. 
\end{example}

This is also the case for the dependence category of posets in \cref{posetdependence}, where the monoidal product $\tri$ is given by the join of posets. Analogous notions of join in other categories admit the same structure, and in fact posets belong to a hierarchy of dependence subcategory inclusions.

\begin{example}\label{categoriesdependence}
For categories $\C$ and $\D$, their \emph{join}, $\C \Join \D$, has objects $\ob(\C) \sqcup \ob(\D)$ and morphisms all those in $\C$ and $\D$, along with for each $c$ in $\C$ and $d$ in $\D$ a unique morphism $c \to d$. This is equivalently the collage of the terminal profunctor from $\C$ to $\D$. The join of categories is easily checked to be functorial and associative, have the empty category as a unit, and preserve finite connected limits.

In the resulting dependence category structure on $\smcat$, the functor $\boxtimes_n^P$ sends $n$ small categories $\C_i$ to the category built from $\coprod_i \C_i$ by adding in unique morphisms from every object in $\C_i$ to every object in $\C_j$ when $i < j$ in $P$. When the categories $\C_i$ are posets, this is precisely the action of $\boxtimes_n^P$ on posets. Posets then form a full dependence subcategory of $\smcat$.
\end{example}

\begin{example}\label{ssetsdependence}
The category of simplicial sets, namely functors $\Delta\op \to \smset$ where the simplex category $\Delta$ is a skeleton of the category of finite nonempty ordinals and monotone maps between them, also has a join operation, $\star$. For simplicial sets $X$ and $Y$, their join $X \star Y$ is given by
\begin{equation}\label{eqn.simplicialjoin}
(X \star Y)_n = X_n \sqcup Y_n \sqcup \coprod_{i=0}^{n-1} X_i \times Y_{n-i-1}.
\end{equation}
Intuitively, $X \star Y$ contains disjoint copies of $X$ and $Y$ connected by adding, for each $n$-simplex $x \in X_n$ and each $m$-simplex $y \in Y_m$, an $n+m+1$ simplex which restricts to $x$ on the first $n+1$ vertices and restricts to $y$ on the remaining $m+1$ vertices.

This join operation has all the same properties as those for posets and categories, including the empty simplicial set as a unit, so simplicial sets form a dependence category where $\boxtimes_n^P$ sends $X^1,...,X^n$ to the simplicial set containing their sum along with the connecting simplices from $X^i$ to $X^j$ as in \eqref{eqn.simplicialjoin} when $i < j$ in $P$, as well as higher order connecting simplices: whenever $i_1 < \cdots < i_k$ in $P$, for $x_1 \in X^{i_1}_{m_1},...,x_k \in X^{i_k}_{m_k}$ there is an $(m_1 + \cdots + m_k + 1)$-simplex connecting them in $\boxtimes_n^P(X^1,...,X^n)$. 

It is easily checked that the fully faithful nerve functor from small categories preserves joins and sums, exhibiting $\smcat$ as a full dependence subcategory of simplicial sets.
\end{example}

\begin{remark}
Categories and simplicial sets generalize more than just the dependence category structure on posets; they also form categorical operads in their own right. Letting $\FC_n$ be the category of small categories with objects $\ul n$, one can define operadic composition of categories $\C \circ (\C_1,...,\C_n)$ by starting with $\C_1 \sqcup \cdots \sqcup \C_n$ and adding in a morphism from every object in $\C_i$ to every object in $\C_j$ for each morphism $i \to j$ in $\C$ (with composition of these arrows ``ignoring'' from $\C_i$ and composing with each other according to composition in $\C$). This is equivalent to the Grothendieck construction of the functor $\C \to \prof$ sending $i$ to the category $\C_i$ and each morphism $i \to j$ in $\C$ to the terminal profunctor from $\C_i$ to $\C_j$.

When $\C$ is a poset, this is precisely $\boxtimes_n^\C(\C_1,...,\C_n)$ from \cref{categoriesdependence}, which shows that there is a fullly faithful functor of categorical operads $\FP \to \FC$. As small categories are easily checked to form an $\FC$-pseudoalgebra, this functor of operads induces the dependence structure on small categories.

Somewhat analogously, the categorical operad of simplicial sets has as $\FS_n$ the category of simplicial sets with vertices $\ul n$. To define the operadic composition $X \circ (X^1,...,X^n)$, we again start with the sum $X^1 \sqcup \cdots \sqcup X^n$ and add connecting simplices according to the structure of $X$. Specifically, for each simplex $x \in X_k$ with vertices $i_1,...,i_k$ and each $x_1 \in X^{i_1}_{m_1},...,x_k \in X^{i_k}_{m_k}$ there is an $(m_1 + \cdots + m_k + 1)$-simplex connecting them in $X \circ (X^1,...,X^n)$. 

Similar to \cref{ssetsdependence}, this operad structure is preserved by the nerve functor, resulting in a fully faithful functor of categorical operads $\FC \to \FS$. The evident $\FS$-pseudoalgebra structure on simplicial sets and restriction along the operad functor $\FP \to \FC \to \FS$ then induce the dependence category structure on simplicial sets. Finally $\FS$ also has a full suboperad of simplicial sets which arise from simplicial complexes, which may be of independent interest, whose intersection with $\FC$ is precisely $\FP$.
\end{remark}

\begin{example}
Finally, topological spaces under sum and join form a physical duoidal category which extends to a dependence category, with the resulting operation $\boxtimes_n^P$ similar to that for simplicial sets. The geometric realization functor from simplicial sets to topological spaces preserves this structure (as it preserves sums and joins), though unlike the dependence functors from posets to categories and categories to simplicial sets this one is not fully faithful, so the dependence structure on simplicial sets (and thereby categories and posets) is not inherited from topological spaces. Also unlike the previous examples, the join of topological spaces is symmetric.
\end{example}

%

\chapter{Process Decoration}\label{chap.decoration}

In the introduction we discussed two distinct ways of reasoning categorically about dependence: physical duoidal categories, where any pair of objects can be combined in independent or dependent fashion, and categories of processes, where composition and tensors of processes only exist when dependence or independence of those processes holds as a property.
The connection between physical duoidal categories and categories of processes lies in decorating the processes with information from a physical duoidal category in a well-behaved manner, which we now make precise.

First, we define our notion of categories of processes: partial monoidal categories. This is inspired by the Causal categories in \cite[Definition 31]{caucats} and the approaches discussed in \cite{promonoidal}, though both of those impose significant additional structure.

\begin{definition}[{\cite[Definition 30]{caucats}}]
A symmetric strict partial monoidal category $\A$ (henceforth called a \emph{category of processes}) is defined analogously to a symmetric strict monoidal category, except the product functor $\otimes \colon \A \times \A \to \A$ is defined only on a full subcategory of $\A \times \A$ which includes $(A,I)$ and $(I,A)$ for all $A$ in $\A$, and which includes both $(A,B)$ and $(A \otimes B,C)$ if and only if it includes both $(B,C)$ and $(A,B \otimes C)$, for all $A,B,C$ in $\A$.
\end{definition}

Given two processes $f,g$, we will treat $g$ as dependent on $f$ when they have a composite $g \circ f$ and independent when they have a tensor product $f \otimes g$. 

\begin{definition}
Given a physical duoidal category $(\C,\yon,\otimes,\tri)$ and a category of processes $(\A,I,\otimes)$, a \emph{decoration} $d$ of $\A$ in $\C$ consists of:
\begin{itemize}
	\item for each $f \colon A \to B$ in $\A$, an object $d_f$ in $\C$;
	\item for each object $A$ in $\A$, an isomorphism $d_{\id_A} \cong \yon$ in $\C$;
	\item for each $f \colon A \to B$ and $f' \colon A' \to B'$ in $\A$ which admit a tensor product, a \emph{productor} morphism $d_f \otimes d_{f'} \to d_{f \otimes f'}$ in $\C$;
	\item and for each $f \colon A \to B$ and $g \colon B \to C$ in $\A$, a \emph{compositor} morphism $d_{g \circ f} \to d_g \;\tri\; d_f$
\end{itemize}
satisfying the evident unit and associativity equations, along with for each $f \colon A \to B$, $g \colon B \to C$, $f' \colon A' \to B'$, and $g' \colon B' \to C'$ in $\A$ admitting all of the relevant tensor products the interchange equation making the diagram in \eqref{eqn.decorationdiag} commute.
\begin{equation}\label{eqn.decorationdiag}
\begin{tikzcd}[column sep=10pt]
& d_{(g \circ f) \otimes (g' \circ f')} \rar[equals] &[10pt] d_{(g \otimes g') \circ (f \otimes f')} \ar{dr} \\
d_{g \circ f} \otimes d_{g' \circ f'} \ar{ur} \ar{dr} & & & d_{g \otimes g'} \;\tri\; d_{f \otimes f'} \\
& (d_g \;\tri\; d_f) \otimes (d_{g'} \;\tri\; d_{f'}) \rar & (d_g \otimes d_{g'}) \;\tri\; (d_f \otimes d_{f'}) \ar{ur}
\end{tikzcd}
\end{equation}
\end{definition}

The intuition behind the directions of the productors and compositors, and their potential non-invertibility, is most easily seen in the example of parallel computing.

\begin{example}[Program runtime]
Let $\A$ be a category whose objects are lists of variables and datatypes, which we call \emph{contexts}, such as
\[ (n : \texttt{Int}, b : \texttt{Bool}, x : \texttt{Real}^2), \]
and morphisms from one list to another are certain algorithms which take as input the variables in the first list and produce as output values for the variables in the second list. Composition is given by sequentially following one program with the other. It is a category of processes where the tensor product of contexts is given by sum, defined only for disjoint pairs of contexts. The tensor product of two algorithms whose input and output contexts are respectively disjoint is the algorithm which performs both computations independently (without requiring a choice of order between the steps of the two algorithms).

A decoration of $\A$ in the physical duoidal category of tropical real numbers can be regarded as an assignment of a runtime to each program, perhaps according to the implementation of the algorithms on some computer (real or abstract). The productors $d_f \max d_{f'} \le d_{f \otimes f'}$ in $\rrnn$ encode how it is impossible to run the independent combination of programs $f,f'$ faster than by running $f$ and $f'$ entirely in parallel, while the compositors $d_{g \circ f} \le d_g + d_f$ encode how any sensible runtime strategy for the sequential composition of $f$ and $g$ should not exceed the sum of their separat runtimes.

But why would $g \circ f$ be able to run faster than the sum of its parts? This is because $f$ and $g$ may themselves be tensor products of independent programs. For instance, if 
\[ f = f \otimes f' \colon (x : T) \otimes (x' : T') \to (y : U) \otimes (y' : U') \]
and
\[ g = g \otimes g' \colon (y : U) \otimes (y' : U') \to (z : V) \otimes (z' : V'), \]
where $d_{f} = 1$, $d_{f'} = 4$, $d_{g} = 3$, and $d_{g'} = 1$, then $d_g + d_f$ is as least 7. However, 
\[ g \circ f = (g \otimes g') \circ (f \otimes f') = (g \circ f) \otimes (g' \circ f') \]
could plausibly be run in as few as $\max(3+1,1+4) = 5$ units of time by running $g \circ f$ and $g' \circ f'$ in parallel.

While this decoration assigns to each program only its runtime, we could also consider decorating programs with objects in some physical duoidal category of ``implementations,'' which could look something like Gantt charts (see for instance \cite{gantt}) which arrange a collection of tasks with durations and dependencies into a chart of at what time each task will be completed, such that multiple tasks may overlap but each must not begin until all of the tasks it depends on have been completed. 
\end{example} 

We say a decoration is \emph{efficient} if the productor $d_f \otimes d_{f'} \to d_{f \otimes f'}$ is an isomorphism for all $f,f'$ which admit a tensor product. In the computational runtime example this corresponds to arbitrary parallel computing capability, as well as the capacity for a computer to recognize when a program has a tensor-decomposition and run each of its $\otimes$-components in parallel.

\begin{example}[Graph-generated decorations]\label{partialgraph}
One of the difficulties in finding efficient decorations is when processes in $\A$ do not admit canonical $\otimes$-decompositions. A class of categories of processes which do have these decompositions are those which are freely generated by a \emph{partial graph}, namely a graph of the form 
\[ G = \left( E \rightrightarrows V \right) \]
equipped with a relation $||$ on $V$. The free symmetric monoidal category on $G$ (defined similarly to the free monoidal category on a graph from \cite[Example 1.2]{shapiro2022enrichment}, further adding in the symmetry isomorphisms) produces a partial monoidal subcategory $\cat{G}$ in which vertices $x,x'$ only admit a tensor product when $x || x'$. In particular, objects in this category are finite lists of vertices in $G$ in which each adjacent pair admit a tensor product. A morphism between two lists, both of fixed length $n$, consists of $n$ paths in $G$ from the vertices in the first list to a permutation of the vertices in the second list. 

For any physical duoidal category $\C$, this partial monoidal category $\cat{G}$ admits an efficient decoration generated by any assignment of $d_f$ in $\C$ to each edge $f$ in $G$. The unique morphism from the monoidal unit to itself is decorated with $\yon$, while each morphism 
\[ f_{1,1},...,f_{1,m_1}; \cdots ; f_{n,1},...,f_{n,m_n} \]
consisting of $n$ disjoint paths, i.e.\ consecutive edges $f_{i,1},\ldots,f_{i,m_i}$ for each $i$. Any permutation on vertices is decorated with
\[ (d_{f_{1,1}} \;\tri\; \cdots \;\tri\; d_{f_{1,m_1}}) \otimes \cdots \otimes (d_{f_{n,1}} \;\tri\; \cdots \;\tri\; d_{f_{n,m_n}}). \]
This assignment is efficient by definition, as each morphism is decorated by the $\otimes$-product of the decorations of its tensor-components, and the compositor $d_{g \circ f} \to d_g \;\tri\; d_f$ is derived from the lax interchanger of $\C$, which makes \eqref{eqn.decorationdiag} commute automatically.
\end{example}

This construction applies, for instance, to a category of programs generated under composition and (partially defined) tensor products by a graph of ``atomic'' programs between single-variable contexts. However, in practice one might want to consider atomic programs between contexts with multiple inputs and outputs. Such programs can be composed according to more complicated string diagrams than the disjoint unions of paths forming the processes in \cref{partialgraph}. 

Recall (from, for instance, \cite[Definition 2.5]{shapely}) that a polygraph consists of a set of vertices along with sets of $m$-to-$n$ arrows between those vertices for each $m,n \in \nn$.

\begin{definition}[Partial polygraphs]
A partial polygraph consists of a polygraph equipped with a relation $||$ on its vertices, such that for each edge from $(x_1,...,x_m)$ to $(y_1,...,y_n)$, $x_i || x_{i+1}$ and $y_j || y_{j+1}$ for all $i = 1,...,m-1$ and $j = 1,...,n-1$.
\end{definition}

Similar to the construction of a freely generated partial monoidal category in \cref{partialgraph}, a partial polygraph generates a partial monoidal category whose objects are \emph{valid} lists of vertices $(x_1,...,x_n)$, meaning that $x_i || x_{i+1}$ for $i = 1,...,n-1$. The morphisms between lists are given by string diagrams of the edges and symmetries, which are ``valid at each stage'' as in \eqref{eqn.stringdiagram}. 
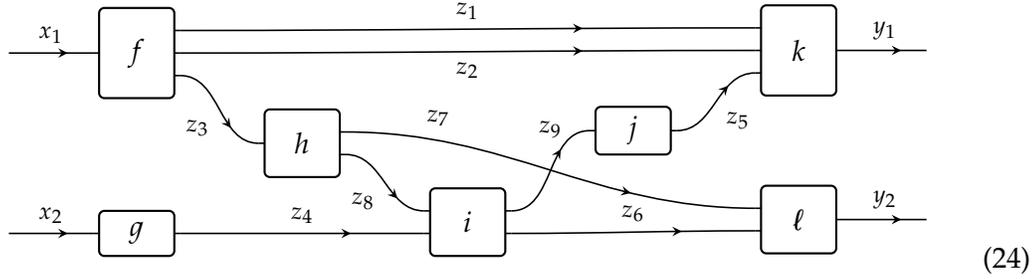
\begin{equation}\label{eqn.stringdiagram}
\begin{tikzpicture}[oriented WD,bb min width =1cm, bbx=1.2cm, bb port sep =2,bb port length=.08cm, bby=.15cm]
  \node [bb={1}{3}] (YY1) {$f$};
  \node [bb={1}{1},below=10 of YY1] (YY2) {$g$};
	\node [bb={1}{2}, below right=1 and 1 of YY1] (XX11) {$h$};
  \node [bb={2}{2}, below right=1 and 1 of XX11] (XX12) {$i$};
  \node [bb={1}{1}, above right=3 and 1 of XX12] (XX13) {$j$};
  \node [bb={3}{1}, above right=1 and 1 of XX13] (YY4) {$k$};
  \node [bb={2}{1}, below=8 of YY4] (YY5) {$\ell$};
  \begin{scope}[white]
    \node [bb={0}{0},fit={($(YY1.north)+(0,2)$) (YY2.south west) (YY4) (YY5) (XX12)}] (NN) {};
  \end{scope}
  \node[coordinate] at (NN.west|-YY1_in1) (NN_in1) {};
  \node[coordinate] at (NN.west|-YY2_in1) (NN_in2) {};
  \node[coordinate] at (NN.east|-YY4_out1) (NN_out1) {};
  \node[coordinate] at (NN.east|-YY5_out1) (NN_out2) {};
  \draw[ar] (NN_in1) to node[above, font=\footnotesize] {$x_1$} (YY1_in1);
  \draw[ar] (NN_in2) to node[above, font=\footnotesize] {$x_2$} (YY2_in1);
  \draw[ar] (YY4_out1) to node[above, font=\footnotesize] {$y_1$} (NN_out1);
  \draw[ar] (YY5_out1) to node[above, font=\footnotesize] {$y_2$} (NN_out2);
  \draw[ar] (YY1_out1) to node[above, font=\footnotesize] {$z_1$} (YY4_in1);
  \draw[ar] (YY1_out2) to node[below, font=\footnotesize] {$z_2$} (YY4_in2);
  \draw[ar] (YY1_out3) to node[below left, font=\footnotesize] {$z_3$} (XX11_in1);
  \draw[ar] (YY2_out1) to node[above, font=\footnotesize] {$z_4$} (XX12_in2);
  \draw[ar] (XX13_out1) to node[below right, font=\footnotesize] {$z_5$} (YY4_in3); 
  \draw[ar] (XX12_out2) to node[above, font=\footnotesize] {$z_6$} (YY5_in2);
  \draw[ar] (XX11_out1) to node[above left=3 and 1, font=\footnotesize] {$z_7$} (YY5_in1);
  \draw[ar] (XX11_out2) to node[below left, font=\footnotesize] {$z_8$} (XX12_in1);
  \draw[ar] (XX12_out1) to node[above=2, font=\footnotesize] {$z_9$} (XX13_in1);
\end{tikzpicture}
\end{equation}
In this diagram the boxes represent edges in the polygraph and the vertices are suppressed, so that the string $z_1$ from $f$ to $k$ can be interpreted as a vertex $z_1$ which is the first target of $f$ and the first source of $k$. Validity at each stage here means that the lists 
\[(z_1,z_2,z_3,z_4),\;(z_1,z_2,z_7,z_8,z_4),\;(z_1,z_2,z_7,z_9,z_6),\;(z_1,z_2,z_9,z_7,z_6),\;(z_1,z_2,z_5,z_6)\]
are all valid, in addition to $(x_1,x_2)$ and $(y_1,y_2)$.

We wish to use a dependence structure on $\C$ to decorate this partial monoidal category by assigning to each edge an object in $\C$ and extending this to all string diagrams using the dependence structure. To do so, we must first associate to each string diagram a finite poset on its edges.

\begin{definition}
Given a string diagram, its \emph{edge poset} has as elements the edges in the string diagram, and its relation generated by setting $f < g$ whenever a target vertex of $f$ is a source vertex of $g$.
\end{definition}

For instance, the edge poset for the string diagram in \eqref{eqn.stringdiagram} is generated by the arrows in \eqref{eqn.edgeposet}
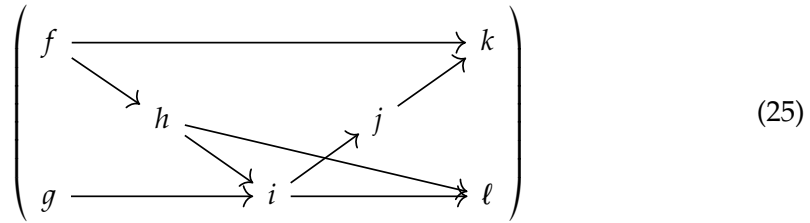
\begin{equation}\label{eqn.edgeposet}
\left(\begin{tikzcd}[row sep=small]
f \ar{rrrr} \ar{dr} & & & & k \\
& h \ar{dr} \ar{drrr} & & j \ar{ur} \\
g \ar{rr} & & i \ar{rr} \ar{ur} & & \ell
\end{tikzcd}\right)
\end{equation}

\begin{remark}
Any finite poset can be realized as the edge poset of some string diagram shape. Given a finite poset $P$ and a choice of two natural numbers $s_i,t_i$ for each of its elements $i$, there is a string diagram shape from $\sum_i s_i$ vertices to $\sum_i t_i$ vertices, where each element of $P$ corresponds to an edge in the string diagram with $s_i$ open source vertices, $t_i$ open target vertices, and whenever $i < j$ is an atomic relationship in $P$, an internal vertex which is a target of edge $i$ and a source of edge $j$. As any two edges are adjacent along at most one vertex, these are precisely the string diagram shapes that can be composed in a symmetric polycategory (\cite[Definition 2.1]{shapely}).
\end{remark}

\begin{theorem}\label{dependencedecoration}
For any partial polygraph $X$, dependence category $\C$, and assignment of an object in $\C$ to each edge in $X$, the partial monoidal category generated by $X$ admits an efficient decoration in $\C$, where a string diagram consisting of edges $f_1,...,f_n$ is decorated by $\boxtimes_n^P(d_{f_1},...,d_{f_n})$ for $P$ the edge poset of the string diagram.
\end{theorem}

\begin{proof}
Efficiency of the decoration is provided by the definition, as a tensor product of string diagrams is simply their disjoint union, as is the corresponding edge poset. This is sufficient as in the dependence category $\C$ we have
\[ \boxtimes_{n_1+n_2}^{P_1 \sqcup P_2} = \boxtimes_{n_1}^{P_1} \otimes \boxtimes_{n_2}^{P_2}. \]

For the compositors, note that for two adjacent string diagrams $f,g$ with edge posets $P,Q$, the edge poset of their composite $g \circ f$ has an identity-on-elements inclusion into $P \Join Q$. This is because the former contains $P \sqcup Q$ by definition, has no additional elements, and any relations between the two go from an edge in $f$ to an edge in $g$. The compositor is then given by the structure map in $\C$ corresponding to this identity-on-elements inclusion. 

Furthermore, by this definition the diagram in \eqref{eqn.decorationdiag} commutes automatically, as all of its morphisms are dependence structure maps in $\C$. In particular, the compositor $d_{(g \otimes g') \circ (f \otimes f')} \to d_{g \otimes g'} \;\tri\; d_{f \otimes f'}$ is precisely the operadic composite of the lax interchanger with the identity maps on the edge posets of $f,f',g,g'$, so up to the productor isomorphisms it agrees with the lax interchanger on the bottom of the diagram by the pseudoalgebra equations.
\end{proof}

\begin{example}\label{efficientparallel}
\cref{dependencedecoration} lets us use the dependence structure on the tropical reals (\cref{tropicaldependence}) as a protocol for running parallelizable programs. Given a polygraph of atomic programs, each from one list of types and variables to another, a string diagram as in \eqref{eqn.stringdiagram} represents a valid way network-composite of those programs with dependencies given by the edge poset. Given a runtime for each atomic program, the decoration given by the dependence structure tells us that this network can be run in the time it takes to run the most expensive dependent sequence of programs in the network. For instance this could be achieved by immediately running as a new parallel thread each atomic program in the network as soon as all its predecessors have terminated.
\end{example}

\begin{example}
Decoration in a dependence category can also describe associating possible outcomes and stimuli to processes in spacetime. The category of processes can be regarded as having objects certain subsets of Minkowski space and morphisms relating to timelike trajectories through spacetime (similar to the main example in \cite{promonoidal}). If certain such objects and processes are selected as generators, then the construction in \cref{dependencedecoration} allows choices of polynomials (\cref{polydependence}) encoding the outcomes and stimuli of these processes to be extended to decorate compound processes built out of them.
\end{example}

\begin{KeepFromToc}

\chapter*{Declarations}

\section*{Author Contributions}

Brandon T. Shapiro and David I. Spivak both wrote and reviewed this manuscript.

\section*{Funding}

This material is based upon work supported by the Air Force Office of Scientific Research under award number FA9550-20-1-0348.

\section*{Data Availability}

There is no additional data associated to the results we present in this paper.

\section*{Competing Interests}

The authors have no competing interests as defined by Springer, or other interests that might be perceived to influence the results and/or discussion reported in this paper.

\end{KeepFromToc}

\printbibliography 
\end{document}